\documentclass[11pt]{amsart}
\usepackage[latin9]{inputenc}
\usepackage[a4paper]{geometry}
\usepackage[active]{srcltx}
\usepackage{mathrsfs}
\usepackage{amstext}
\usepackage{amsthm}
\usepackage{amssymb}
\usepackage[numbers,sort&compress]{natbib}

\makeatletter

\providecommand{\tabularnewline}{\\}

\numberwithin{equation}{section}
\numberwithin{figure}{section}

\usepackage[all]{xy}
\usepackage{indentfirst}
\usepackage{amscd}\usepackage{tikz}
\usepackage{tikz-cd}

\allowdisplaybreaks[2]

\newtheorem{theorem}{Theorem}[section]
\newtheorem{proposition}[theorem]{Proposition}\newtheorem{lemma}[theorem]{Lemma}\theoremstyle{definition}
\newtheorem{remark}[theorem]{Remark}\newtheorem{definition}[theorem]{Definition}\newtheorem{example}[theorem]{Example}\numberwithin{equation}{section}

\makeatother

\begin{document}
\title{Double Extensions of Multiplicative Restricted Hom-Lie Algebras}
\author{Dan Mao}
\address{D. Mao: School of Mathematics and Statistics, Northeast Normal University,
Changchun 130024, China}
\email{danmao678@nenu.edu.cn}
\author{Zeyu Hao}
\address{Z. Hao: School of Mathematics and Statistics, Northeast Normal University,
Changchun 130024, China}
\email{haozy\_ hb@163.com}
\author{Liangyun Chen$^{*}$}
\address{L. Chen: School of Mathematics and Statistics, Northeast Normal University,
Changchun 130024, China}
\email{chenly640@nenu.edu.cn}
\thanks{{*}Corresponding author.}
\thanks{\emph{MSC}(2020). 17A60, 17B61.}
\thanks{\emph{Key words and phrases}. Restricted Hom-Lie algebras, Double
extensions, Quadratic Hom-Lie algebras, Restricted derivations. }
\thanks{Supported by NSF of Jilin Province (No. YDZJ202201ZYTS589), NNSF of
China (Nos. 12271085, 12071405) and the Fundamental Research Funds for the Central Universities.}
\begin{abstract}
In this paper, we study the double extension of a restricted quadratic
Hom-Lie algebra $(V,[\cdot,\cdot]_{V},\alpha_{V},B_{V})$, which is
an enlargement of $V$ by means of a central extension and a restricted
derivation $\mathscr{D}$. In particular, we prove that the double
extension of a restricted quadratic Hom-Lie algebra $V$ with a $\mathscr{D}$-invariant
bilinear form $B_{V}$ is restricted. Conversely, any irreducible
restricted quadratic Hom-Lie algebra with nonzero center is proved
to be the double extension of another restricted quadratic Hom-Lie
algebra.

\tableofcontents{}
\end{abstract}

\maketitle

\section{Introduction}

The concept of a restricted Lie algebra is introduced by Jacobson
in \citep{Jac,Jac2}. Because of its important role in modular theory,
restrictedness theory attracts great attention \citep{BM2,BMPZ,DL,EF,Farn,GC,Hod,Sha,SLGW,WZ,SF,Hoc,Fel}.
Particularly, the notion of restricted Hom-Lie algebras is introduced
in \citep{GC} and then the restrictedness theory of Hom-Lie superalgebras
is studied in \citep{Sha}. In addition, \citep{BM2} proposes an
alternative definition of restricted Hom-Lie algebras to study the
queerification of a restricted Hom-Lie algebra in characteristic 2.

The notion of the double extension of Lie algebras over $\mathbb{R}$
is introduced by Medina and Revoy \citep{MR}, which is as follows:
Let $(V,[\cdot,\cdot]_{V})$ be a Lie algebra with a nondegenerate
invariant symmetric bilinear form $B_{V}$ and $\mathscr{D}\in\mathrm{Der}(V)$
a derivation. If $B_{V}$ is $\mathscr{D}$-invariant, i.e., $B_{V}(\mathscr{D}(u),v)+B_{V}(u,\mathscr{D}(v))=0$,
$\forall u,v\in V$, then there exists a quadratic algebra structure
on the vector space $\mathfrak{g}:=\mathscr{E}\oplus V\oplus\mathscr{E}^{*}$,
where $\mathscr{E}=\mathbb{R}e$, $\mathscr{E}^{*}=\mathbb{R}e^{*}$
for $e^{*}=\mathscr{D}$, the bracket $[\cdot,\cdot]:\mathfrak{g}\times\mathfrak{g}\rightarrow\mathfrak{g}$
is defined as follows (for any $u,v\in V$):
\[
[e,\mathfrak{g}]=0,\:[u,v]=[u,v]_{V}+B_{V}(\mathscr{D}(u),v)e,\:[e^{*},u]=\mathscr{D}(u)
\]
and the bilinear form $B:\mathfrak{g}\times\mathfrak{g}\rightarrow\mathbb{R}$
is defined as follows (for any $u,v\in V$):
\[
B(u,v)=B_{V}(u,v),\:B(u,e^{*})=B(u,e)=0,\:B(e^{*},e)=1,\:B(e,e)=B(e^{*},e^{*})=0.
\]
The quadratic Lie algebra $(\mathfrak{g},[\cdot,\cdot],B)$ is called
a double extension of $V$. The study of double extensions of Lie
algebras attracts great attention \citep{ARS,RS,RS2,JA,FS}. In \citep{BB},
the notion is generalized to quadratic Lie superalgebras, which initiates
the research of the double extension of Lie superalgebras \citep{ABB,ABB2,Be,Be2,BeBou,BLS,BB2}.
Moreover, (involutive) double extensions of Hom-Lie algebras are given
in \citep{BM}; double extensions of restricted Lie (super)algebras
are studied in \citep{BBH} and symplectic double extensions for restricted
quasi-Frobenius Lie superalgebras have been investigated in \citep{BEM}.

This paper generalizes the double extension theories of Hom-Lie algebras
\citep{BM} and restricted Lie algebras \citep{BBH} in the same time
and studies the double extension of restricted Hom-Lie algebras in
characteristic $p$. In particular, we give the double extension theorem
of a restricted Hom-Lie algebra and its converse theorem. We study
the adapted isomorphism between double extensions of the same multiplicative
quadratic Hom-Lie algebra and the equivalence class of $p$-structures
on double extensions.

The paper is organized as follows. In Sec. 2, some necessary definitions
and conclusions are recalled or introduced, including quadratic Hom-Lie
algebras, $p$-structures, restricted derivations and so on. In Sec.
3, we study the double extension of restricted Hom-Lie algebras for
$p=2$. We prove that the $2$-structure $[2]_{V}$ on the restricted
quadratic Hom-Lie algebra $(V,[\cdot,\cdot]_{V},\alpha_{V},B_{V})$
can be extended to its any double extension, i.e., the double extension
of a restricted quadratic Hom-Lie algebra $V$ with a $\mathscr{D}$-invariant
bilinear form $B_{V}$ is restricted (Theorem 3.4). We also give the
adapted isomorphism between double extensions of the same multiplicative
quadratic Hom-Lie algebra (Theorem 3.8) and the equivalence class
of $2$-structures on double extensions (Theorem 3.9). In Sec. 4,
we study the double extension of restricted Hom-Lie algebras for $p>2$.
We give the double extension of a multiplicative quadratic Hom-Lie
algebra (Theorem 4.1), $p$-structures on double extensions (Theorem
4.4) and the adapted isomorphism between double extensions of the
same multiplicative quadratic Hom-Lie algebra. We also study the equivalence
class of $p$-structures on double extensions (Theorem 4.13). In Sec.
5, we provide several examples of double extensions.

Throughout the paper, we study the finite-dimensional multiplicative
Hom-Lie algebras in prime characteristic $p$.

\section{Preliminaries}

\begin{definition} (see \citep{HLS}) A Hom-Lie algebra over $\mathbb{F}$
is a triple ($\mathfrak{g},[\cdot,\cdot],\alpha$) consisting of a
vector space $\mathfrak{g}$, a map $\alpha\in\mathrm{End}(\mathfrak{g})$
and a bilinear bracket $[\cdot,\cdot]:\mathfrak{g}\times\mathfrak{g}\rightarrow\mathfrak{g}$
satisfying:
\begin{align*}
[x,y] & =-[y,x],\:[\alpha(x),[y,z]]+\circlearrowleft(x,y,z)=0,\:\alpha([x,y])=[\alpha(x),\alpha(y)],
\end{align*}
for all $x,y,z\in\mathfrak{g}$.

\end{definition}

\begin{remark} (see \citep{BM2}) If $\mathrm{Char}(\mathbb{F})=2$,
then the first condition above is replaced by $[x,x]=0$, for all
$x\in\mathfrak{g}$.

\end{remark}

\begin{definition} (see \citep{BM}) Let ($\mathfrak{g},[\cdot,\cdot],\alpha$)
be a Hom-Lie algebra, a subspace $\mathfrak{h}\subset\mathfrak{g}$
is called an ideal of $\mathfrak{g}$ if $\alpha(x)\in\mathfrak{h}$
and $[x,y]\in\mathfrak{h}$ for all $x\in\mathfrak{h},y\in\mathfrak{g}$.

\end{definition}

\begin{definition} (see \citep{BM}) Let ($\mathfrak{g},[\cdot,\cdot],\alpha$)
be a Hom-Lie algebra. It is called involutive if $\alpha$ is an involution,
that is $\alpha^{2}=\mathrm{id}$.

\end{definition}

\begin{definition} (see \citep{She}) A representation of a Hom-Lie
algebra $(\mathfrak{g},[\cdot,\cdot]_{\mathfrak{g}},\alpha)$ is a
triple $(V,[\cdot,\cdot]_{V},\beta)$, where $V$ is a vector space,
$\beta\in\mathrm{gl}(V)$ and $[\cdot,\cdot]_{V}$ is the action of
$\mathfrak{g}$ on $V$ such that (for all $x,y\in\mathfrak{g}$ and
for all $v\in V$ ):

(1) $[\alpha(x),\beta(v)]_{V}=\beta([x,v]_{V})$,

(2) $[[x,y]_{\mathfrak{g}},\beta(v)]_{V}=[\alpha(x),[y,v]_{V}]_{V}+[\alpha(y),[x,v]_{V}]_{V}$
.

\end{definition}

If we put $\rho_{\beta}:=[\cdot,\cdot]_{V}$, then the above two equations
can be written as follows (for all $x,y\in\mathfrak{g}$):
\[
\rho_{\beta}(\alpha(x))\circ\beta=\beta\circ\rho_{\beta}(x),\:\rho_{\beta}([x,y]_{\mathfrak{g}})\circ\beta=\rho_{\beta}(\alpha(x))\rho_{\beta}(y)+\rho_{\beta}(\alpha(y))\rho_{\beta}(x).
\]

\begin{definition} (see \citep{BM}) Let ($\mathfrak{g},[\cdot,\cdot],\alpha$)
be a Hom-Lie algebra over $\mathbb{F}$.

(1) A symmetric nondegenerate bilinear form $B:\mathfrak{g}\times\mathfrak{g}\rightarrow\mathbb{F}$
is said to be invariant if it satisfies

\[
B([x,y],z)=B(x,[y,z]),\:\mathrm{for}\:\mathrm{any}\:x,y,z\in\mathfrak{g}.
\]

(2) The quadruple $(\mathfrak{g},[\cdot,\cdot],\alpha,B)$ is called
a quadratic Hom-Lie algebra if the symmetric nondegenerate invariant
bilinear form $B:\mathfrak{g}\times\mathfrak{g}\rightarrow\mathbb{F}$
satisfies
\begin{equation}
B(\alpha(x),y)=B(x,\alpha(y)),\:\mathrm{for}\:\mathrm{any}\:x,y\in\mathfrak{g}.
\end{equation}

\end{definition}

\emph{Notation}. Let $(\mathfrak{g},[\cdot,\cdot],\alpha,B)$ be a
quadratic Hom-Lie algebra. We denote by $\mathrm{End}_{a}(\mathfrak{g},B)$
the set of skew-symmetric endomorphisms of $\mathfrak{g}$ with respect
of $B$, that is endomorphisms $\phi:\mathfrak{g}\rightarrow\mathfrak{g}$
such that $B(\phi(x),y)=-B(x,\phi(y)),\:\mathrm{for}\:\mathrm{any}\:x,y\in\mathfrak{g}$.
Similarly, We denote by $\mathrm{End}_{s}(\mathfrak{g},B)$ the set
of symmetric endomorphisms of $\mathfrak{g}$ with respect of $B$,
that is endomorphisms $\phi:\mathfrak{g}\rightarrow\mathfrak{g}$
such that $B(\phi(x),y)=B(x,\phi(y)),\:\mathrm{for}\:\mathrm{any}\:x,y\in\mathfrak{g}$.

\begin{definition} Let $(\mathfrak{g},[\cdot,\cdot],\alpha,B)$ be
a quadratic Hom-Lie algebra.

(1) An ideal $\mathfrak{h}$ of $\mathfrak{g}$ is said to be nondegenerate
if $B|_{\mathfrak{h}\times\mathfrak{h}}$ is nondegenerate.

(2) The quadratic Hom--Lie algebra is said to be irreducible (or
$B$-irreducible) if $\mathfrak{g}$ does not contain any nondegenerate
ideal $\mathfrak{h}$ such that $\mathfrak{h}\neq\{0\}$ and $\mathfrak{h}\neq\mathfrak{g}$.

(3) Let $\mathfrak{h}$ be an ideal of $\mathfrak{g}$. Denote by
$\mathfrak{h}^{\perp}$ the orthogonal space with respect to $B$
of $\mathfrak{h}$, which is defined by
\[
\mathfrak{h}^{\perp}=\{x\in\mathfrak{g}:B(x,y)=0,\:\forall y\in\mathfrak{h}\}.
\]

\end{definition}

Denote by $\mathfrak{z}(\mathfrak{g})$ the center of Hom-Lie algebra
($\mathfrak{g},[\cdot,\cdot],\alpha$), that is $\mathfrak{z}(\mathfrak{g})=\{x\in\mathfrak{g}:[x,y]=0,\forall y\in\mathfrak{g}\}$.
The following lemma is proved over an algebraically closed field of
characteristic zero \citep{BM}, which still holds over an algebraically
closed field of characteristic $p$.

\begin{lemma} Let $(\mathfrak{g},[\cdot,\cdot],\alpha,B)$ be a multiplicative
quadratic Hom-Lie algebra in characteristic $p$. Then the center
$\mathfrak{z}(\mathfrak{g})$ is an ideal of $\mathfrak{g}$.

\end{lemma}

\begin{definition} (see \citep{SF}) Let $\mathfrak{g}$ be a Lie
algebra over $\mathbb{F}$. A mapping $[p]:\mathfrak{g}\rightarrow\mathfrak{g},x\mapsto x^{[p]}$
is called a $p$-mapping if

(1) $\mathrm{ad}x^{[p]}=(\mathrm{ad}x)^{p}$, for all $x\in\mathfrak{g}.$

(2) $(kx)^{[p]}=k^{p}x^{[p]}$, for all $x\in\mathfrak{g}$ and all
$k\in\mathbb{F}$.

(3) $(x+y)^{[p]}=x^{[p]}+y^{[p]}+\sum_{i=1}^{p-1}s_{i}(x,y)$, where
$(\mathrm{ad}(x\otimes X+y\otimes1))^{p-1}(x\otimes1)=\sum_{i=1}^{p-1}is_{i}(x,y)\otimes X^{i-1}$,
in $\mathfrak{g}\otimes_{\mathbb{F}}\mathbb{F}[X]$, for all $x,y\in\mathfrak{g}$.
The pair $(\mathfrak{g},[p])$ is referred to as a restricted Lie
algebra.

\end{definition}

The notion of restricted Hom-Lie algebras introduced by Bouarroudj
and Makhlouf \citep{BM2} will be used in this paper, which is as
follows:

\begin{definition} (see \citep{BM2}) Let $\mathfrak{g}$ be a multiplicative
Hom-Lie algebra in characteristic $p$ with a twist $\alpha$. A mapping
$[p]:\mathfrak{g}\rightarrow\mathfrak{g}$ , $x\mapsto x^{[p]}$ is
called a $p$-structure of $\mathfrak{g}$ and $\mathfrak{g}$ is
said to be restricted if

(R1) $\mathrm{ad}(x^{[p]})\circ\alpha^{p-1}=\mathrm{ad}(\alpha^{p-1}(x))\circ\mathrm{ad}(\alpha^{p-2}(x))\circ\cdot\cdot\cdot\circ\mathrm{ad}(x)$,
for all $x\in\mathfrak{g}$;

(R2) $(kx)^{[p]}=k^{p}x^{[p]}$, for all $x\in\mathfrak{g}$ and for
all $k\in\mathbb{F}$;

(R3) $(x+y)^{[p]}=x^{[p]}+y^{[p]}+\sum_{i=1}^{p-1}s_{i}^{\mathfrak{g}}(x,y)$,
where the $s_{i}^{\mathfrak{g}}(x,y)$ can be obtained from $\mathrm{ad}(\alpha^{p-2}(kx+y))\circ\mathrm{ad}(\alpha^{p-3}(kx+y))\circ\cdot\cdot\cdot\circ\mathrm{ad}(kx+y)(x)=\sum_{i=1}^{p-1}is_{i}^{\mathfrak{g}}(x,y)k^{i-1}$.

\end{definition}

\begin{lemma} (see \citep{MGC}) Let $(e_{j})_{j\in J}$ be a basis
of the Hom-Lie algebra $(\mathfrak{g},[\cdot,\cdot],\alpha)$ such
that there are $y_{j}\in\mathfrak{g}$ with
\[
\mathrm{ad}(\alpha^{p-1}(e_{j}))\circ\mathrm{ad}(\alpha^{p-2}(e_{j}))\circ...\circ\mathrm{ad}(e_{j})=\mathrm{ad}(y_{j})\circ\alpha^{p-1}.
\]
Then there exists exactly one $p$-structure $[p]:\mathfrak{g}\rightarrow\mathfrak{g}$
such that $e_{j}^{[p]}=y_{j},$ for all $j\in J$.

\end{lemma}

\begin{definition} An ideal $\mathfrak{h}$ of the restricted Hom-Lie
algebra $(\mathfrak{g},[\cdot,\cdot],\alpha,[p])$ in characteristic
$p$ is called a $p$-ideal if $x^{[p]}\in\mathfrak{h}$ for all $x\in\mathfrak{h}$.

\end{definition}

\begin{definition} (see \citep{GC}) Let $(L,[\cdot,\cdot]_{L},\alpha,[p]_{1})$
and $(\widetilde{L},[\cdot,\cdot]_{\widetilde{L}},\widetilde{\alpha},[p]_{2})$
be restricted Hom-Lie algebras. A homomorphism $\pi:L\rightarrow\widetilde{L}$
is said to be restricted if

\[
\pi(x^{[p]_{1}})=\pi(x)^{[p]_{2}},\:\mathrm{for\:every\:}x\in L.
\]

\end{definition}

Similarly, an isomorphism $\pi:L\rightarrow\widetilde{L}$ is said
to be restricted if $\pi(x^{[p]_{1}})=\pi(x)^{[p]_{2}}$, for every
$x\in L$.

\begin{definition} (see \citep{She}) Let ($\mathfrak{g},[\cdot,\cdot],\alpha$)
be a multiplicative Hom-Lie algebra. For any nonnegative integer $k$,
a linear map $\mathscr{D}:\mathfrak{g}\rightarrow\mathfrak{g}$ is
called an $\alpha^{k}$-derivation of $\mathfrak{g}$ if
\[
\mathscr{D}\circ\alpha=\alpha\circ\mathscr{D}
\]

and
\[
\mathscr{D}([x,y])=[\mathscr{D}(x),\alpha^{k}(y)]+[\alpha^{k}(x),\mathscr{D}(y)],\:\forall x,y\in\mathfrak{g}.
\]

\end{definition}

Let $\mathrm{Der}_{\alpha^{k}}(\mathfrak{g})$ denote the set of $\alpha^{k}$-derivations
of the multiplicative Hom-Lie algebra ($\mathfrak{g},[\cdot,\cdot],\alpha$).
Then $\mathrm{Der}(\mathfrak{g})=\oplus_{k\geq0}\mathrm{Der}_{\alpha^{k}}(\mathfrak{g})$
is the derivation algebra of $\mathfrak{g}$.

\begin{definition} (see \citep{BBH}) Let ($\mathfrak{g},[\cdot,\cdot],\alpha,B$)
be a quadratic Hom-Lie algebra over $\mathbb{F}$ and $\mathscr{D}\in\mathrm{Der}(\mathfrak{g})$.
The bilinear form $B$ is said to be $\mathscr{D}$-invariant if for
any $x,y\in\mathfrak{g}$,

(1) $p=2$, $B(x,\mathscr{D}(x))=0$.

(2) $p>2$, $B(\mathscr{D}(x),y)+B(x,\mathscr{D}(y))=0$.

\end{definition}

\begin{definition} Let ($\mathfrak{g},[\cdot,\cdot],\alpha$) be
a restricted multiplicative Hom-Lie algebra and $[p]_{\mathfrak{g}}$
a $p$-structure on $\mathfrak{g}$. $\mathscr{D}\in\mathrm{Der}(\mathfrak{g})$
is said to be restricted if
\[
\mathscr{D}(x^{[p]_{\mathfrak{g}}})=\mathrm{ad}(\alpha^{p-1}(x))\circ\mathrm{ad}(\alpha^{p-2}(x))\circ\cdot\cdot\cdot\circ\mathrm{ad}(\alpha(x))(\mathscr{D}(x)),\:\forall x\in\mathfrak{g}.
\]

\end{definition}

Denote by $\mathrm{Der}^{p}(\mathfrak{g})$ the set of all restricted
derivations of $\mathfrak{g}$.

\begin{definition} Let ($\mathfrak{g},[\cdot,\cdot],\alpha$) be
a restricted multiplicative Hom-Lie algebra in characteristic $p$.
$\mathscr{D}\in\mathrm{Der}^{p}(\mathfrak{g})$ is said to have $p$-property
if there exists $\xi\in\mathbb{F}$ and $a_{0}\in\mathfrak{g}$ such
that

\begin{equation}
\mathscr{D}^{p}=\xi\mathscr{D}\circ\alpha^{p-1}+\mathrm{ad}(a_{0})\circ\alpha^{p-1},\:\mathscr{D}(a_{0})=0.
\end{equation}

\end{definition}

As is known, in the case of (restricted) Lie algebras, double extension
is closely related to (restricted) cohomology (see \citep{BBH,RS2}).
Details are as follows:

Denote by $\mathfrak{h}=\mathscr{E}^{*}\oplus\mathfrak{g}\oplus\mathscr{E}$
the double extension of a Lie algebra $\mathfrak{g}$, which involves
three ingredients:

(1) an outer derivation $\mathscr{D}$ of $\mathfrak{g}$,

(2) a nondegenerade symmetric invariant bilinear form $B_{\mathfrak{g}}$
which is $\mathscr{D}$-invariant,

(3) the central extension $\mathfrak{g}_{e}$ of $\mathfrak{g}$ given
by $(x,y)\mapsto B_{\mathfrak{g}}(\mathscr{D}(x),y)e=\theta(x,y)e$,
$\forall x,y\in\mathfrak{g}$, where $\theta\in(\wedge^{2}\mathfrak{g})^{*}$
is a 2-cocycle.

It implies that the central extension $\mathfrak{g}_{e}$ of $\mathfrak{g}$
is determined by a 2-cocycle $\theta\in(\wedge^{2}\mathfrak{g})^{*}$.

On the other hand, let $\mathfrak{out}(\mathfrak{g})=\mathrm{Der}(\mathfrak{g})/\mathrm{ad_{\mathfrak{g}}}$
denote the space of outer derivations of $\mathfrak{g}$, then $\mathfrak{out}(\mathfrak{g})\simeq H^{1}(\mathfrak{g};\mathfrak{g})$,
where $H^{1}(\mathfrak{g};\mathfrak{g})$ is the space of $1^{th}$
cohomology classes of $\mathfrak{g}$. Furthermore, suppose that $\mathfrak{g}$
is a restricted Lie algebra. Then the double extension of $\mathfrak{g}$
involves another ingredient, that is, a $p$-mapping $[p]$ on $\mathfrak{g}$.
Now, the outer derivation $\mathscr{D}$ is required to be restricted
with respect to $[p]$, that is, $\mathscr{D}\in\mathfrak{out}^{p}(g)=\mathrm{Der}^{p}(\mathfrak{g})/\mathrm{ad_{\mathfrak{g}}}$.
In fact, $\mathfrak{out}^{p}(\mathfrak{g})\simeq H_{res}^{1}(\mathfrak{g};\mathfrak{g})$,
where $H_{res}^{1}(\mathfrak{g};\mathfrak{g})$ is the space of restricted
$1^{th}$ cohomology classes of $\mathfrak{g}$. Particularly, if
$\mathfrak{g}$ is simple, then $H^{1}(\mathfrak{g};\mathfrak{g})\simeq H_{res}^{1}(\mathfrak{g};\mathfrak{g})$.

Consequently, in the study of double extensions of a given restricted
Lie algebra $\mathfrak{g}$, the restricted outer derivations of $\mathfrak{g}$
can be captured by computing $H_{res}^{1}(\mathfrak{g};\mathfrak{g})$
or $H^{1}(\mathfrak{g};\mathfrak{g})$ (if $\mathfrak{g}$ is simple).

However, in the case of Hom-Lie algebra, the cohomology structures
on Hom-Lie algebras in characteristic $0$ defined in \citep{MS}
and \citep{AEM} are not fit for the study of double extensions of
Hom-Lie algebras.

\section{Double Extensions of Restricted Hom-Lie Algebras for $p=2$}

In this section, $\mathbb{F}$ is an algebraically closed field of
characteristic $2$ and $(V,[\cdot,\cdot]_{V},\alpha_{V})$ is a Hom-Lie
algebra over $\mathbb{F}$. The following theorem characterizes the
double extension of a quadratic Hom-Lie algebra $(V,[\cdot,\cdot]_{V},\alpha_{V},B_{V})$
by means of a central extension and a derivation.

\begin{theorem} (Double Extension Theorem) Let $(V,[\cdot,\cdot]_{V},\alpha_{V},B_{V})$
be an involutive quadratic Hom-Lie algebra, $\mathscr{D}\in\mathrm{Der}_{\alpha_{V}}(V)$
which makes $B_{V}$ $\mathscr{D}$-invariant, $\mathscr{E}=\mathrm{Span}\{e\}$
and $\mathscr{E}^{*}=\mathrm{Span}\{e^{*}\}$, where $e^{*}=\mathscr{D}$.

(1) Suppose that there exist $x_{0}\in V,\lambda\in\mathbb{F}$ such
that
\begin{align}
\lambda\mathscr{D}+\mathrm{ad}_{V}(x_{0}) & =\mathscr{D},\\
\alpha_{V}(\mathscr{D}(x_{0})) & =\mathscr{D}(x_{0}),\\
\alpha_{V}\circ\mathscr{D}^{2}+\mathscr{D}^{2}\circ\alpha_{V} & =\mathrm{ad}_{V}(\mathscr{D}(x_{0})),
\end{align}
where $\mathrm{ad}_{V}(x):=[x,\cdot]_{V}$, for $x\in V$, then there
exists a multiplicative quadratic Hom-Lie algebra structure on $L:=\mathscr{E}^{*}\oplus V\oplus\mathscr{E}$,
where the bracket $[\cdot,\cdot]:L\times L\rightarrow L$ is defined
by
\begin{align*}
[x,y] & =[x,y]_{V}+B_{V}(\mathscr{D}(x),y)e,\:\forall x,y\in V,\\{}
[e^{*},x] & =\mathscr{D}(x),\:\forall x\in V,\\{}
[e,z] & =0,\:\forall z\in L,
\end{align*}
and the linear map $\alpha:L\rightarrow L$ is defined as follows:
\[
\alpha(x)=\alpha_{V}(x)+B_{V}(x_{0},x)e,\:\forall x\in V,\:\alpha(e^{*})=\lambda e^{*}+x_{0}+\lambda_{0}e,\:\alpha(e)=\lambda e,
\]
where $\lambda_{0}\in\mathbb{F}$.

The symmetric nondegenerate invariant bilinear form $B:L\times L\rightarrow\mathbb{F}$
is defined by
\begin{align*}
B(x,y) & =B_{V}(x,y),\:B(x,e^{*})=B(x,e)=0,\:\forall x,y\in V,\\
B(e^{*},e) & =1,\:B(e,e)=0.
\end{align*}

(2) The twist map $\alpha$ is invertible if and only if $\lambda\neq0$.
Furthermore, $\alpha$ is an involution if and only if

\begin{equation}
\lambda^{2}=1,\:\alpha_{V}(x_{0})=\lambda x_{0},\:B_{V}(x_{0},x_{0})=0.
\end{equation}

\end{theorem}
\begin{proof}
(1) First, for any $f=je^{*}+v+ke\in L$, we have
\[
[f,f]=[je^{*}+v+ke,je^{*}+v+ke]=j[e^{*},v]+j[v,e^{*}]=2j\mathscr{D}(v)=0.
\]

Second, the bracket $[\cdot,\cdot]:L\times L\rightarrow L$ will be
proved to satisfy the Hom-Jacobi identity, that is, for any $f,g,h\in L$,
\[
[\alpha(h),[f,g]]+[\alpha(f),[g,h]]+[\alpha(g),[h,f]]=0.
\]

If $h=e,$ then we obtain
\[
[\alpha(e),[f,g]]+[\alpha(f),[g,e]]+[\alpha(g),[e,f]]=\lambda[e,[f,g]]=0.
\]

If $h=e^{*}$, then there exist three cases as follows:

(i) If $f=e$ (or the way around $g=e$), then the Hom-Jacobi identity
holds since $e\in\mathfrak{\mathscr{\mathfrak{z}}}(\mathfrak{g})$
and $\alpha(e)=\lambda e$.

(ii) If $f=e^{*}$ (or the way around $g=e^{*}$), then
\[
[\alpha(e^{*}),[e^{*},g]]+[\alpha(e^{*}),[g,e^{*}]]+[\alpha(g),[e^{*},e^{*}]]=2[\alpha(e^{*}),[e^{*},g]]=0.
\]

(iii) If $f,g\in V$, then we have
\begin{align*}
 & [\alpha(e^{*}),[f,g]]+[\alpha(f),[g,e^{*}]]+[\alpha(g),[e^{*},f]]\\
= & [\lambda e^{*}+x_{0}+\lambda_{0}e,[f,g]_{V}+B_{V}(\mathscr{D}(f),g)e]\\
 & +[\alpha_{V}(f)+B_{V}(x_{0},f)e,\mathscr{D}(g)]+[\alpha_{V}(g)+B_{V}(x_{0},g)e,\mathscr{D}(f)]\\
= & [\lambda e^{*}+x_{0},[f,g]_{V}]+[\alpha_{V}(f),\mathscr{D}(g)]+[\alpha_{V}(g),\mathscr{D}(f)]\\
= & \lambda\mathscr{D}([f,g]_{V})+[x_{0},[f,g]_{V}]+[\alpha_{V}(f),\mathscr{D}(g)]+[\alpha_{V}(g),\mathscr{D}(f)]\\
= & \lambda\mathscr{D}([f,g]_{V})+[x_{0},[f,g]_{V}]_{V}+B_{V}(\mathscr{D}(x_{0}),[f,g]_{V})e+[\alpha_{V}(f),\mathscr{D}(g)]_{V}\\
 & +B_{V}(\mathscr{D}(\alpha_{V}(f)),\mathscr{D}(g))e+[\alpha_{V}(g),\mathscr{D}(f)]_{V}+B_{V}(\mathscr{D}(\alpha_{V}(g)),\mathscr{D}(f))e\\
= & (\lambda\mathscr{D}+\mathrm{ad}_{V}(x_{0}))([f,g]_{V})+[\alpha_{V}(f),\mathscr{D}(g)]_{V}+[\mathscr{D}(f),\alpha_{V}(g)]_{V}\\
 & +(B_{V}([\mathscr{D}(x_{0}),f]_{V},g)+B_{V}(\mathscr{D}^{2}\circ\alpha_{V}(f),g)+B_{V}(g,\alpha_{V}\circ\mathscr{D}^{2}(f)))e\\
= & 2(\lambda\mathscr{D}+\mathrm{ad}_{V}(x_{0}))([f,g]_{V})+B_{V}((\mathrm{ad}_{V}(\mathscr{D}(x_{0}))+\mathscr{D}^{2}\circ\alpha_{V}+\alpha_{V}\circ\mathscr{D}^{2})(f),g)e\\
= & 2B_{V}(\mathrm{ad}_{V}(\mathscr{D}(x_{0}))(f),g)e\\
= & 0.
\end{align*}

If $f,g,h\in V$, then
\begin{align*}
 & [\alpha(h),[f,g]]+[\alpha(f),[g,h]]+[\alpha(g),[h,f]]\\
= & [\alpha_{V}(h)+B_{V}(x_{0},h)e,[f,g]_{V}+B_{V}(\mathscr{D}(f),g)e]\\
 & +[\alpha_{V}(f)+B_{V}(x_{0},f)e,[g,h]_{V}+B_{V}(\mathscr{D}(g),h)e]\\
 & +[\alpha_{V}(g)+B_{V}(x_{0},g)e,[h,f]_{V}+B_{V}(\mathscr{D}(h),f)e]\\
= & [\alpha_{V}(h),[f,g]_{V}]+[\alpha_{V}(f),[g,h]_{V}]+[\alpha_{V}(g),[h,f]_{V}]\\
= & ([\alpha_{V}(h),[f,g]_{V}]_{V}+[\alpha_{V}(f),[g,h]_{V}]_{V}+[\alpha_{V}(g),[h,f]_{V}]_{V})\\
 & +(B_{V}(\mathscr{D}(\alpha_{V}(h)),[f,g]_{V})+B_{V}(\mathscr{D}(\alpha_{V}(f)),[g,h]_{V})+B_{V}(\mathscr{D}(\alpha_{V}(g)),[h,f]_{V}))e\\
= & (B_{V}(\alpha_{V}(h),\mathfrak{\mathscr{D}}([f,g]_{V}))+B_{V}(\alpha_{V}(f),\mathscr{D}([g,h]_{V}))+B_{V}(\alpha_{V}(g),\mathscr{D}([h,f]_{V})))e\\
= & (B_{V}(\alpha_{V}(h),[\mathfrak{\mathscr{D}}(f),\alpha_{V}(g)]_{V}+[\alpha_{V}(f),\mathfrak{\mathscr{D}}(g)]_{V})\\
 & +B_{V}(\alpha_{V}(f),[\mathfrak{\mathscr{D}}(g),\alpha_{V}(h)]_{V}+[\alpha_{V}(g),\mathfrak{\mathscr{D}}(h)]_{V})\\
 & +B_{V}(\alpha_{V}(g),[\mathfrak{\mathscr{D}}(h),\alpha_{V}(f)]_{V}+[\alpha_{V}(h),\mathfrak{\mathscr{D}}(f)]_{V}))e\\
= & (B_{V}(\alpha_{V}(h),[\mathfrak{\mathscr{D}}(f),\alpha_{V}(g)]_{V})+B_{V}(\alpha_{V}(h),[\alpha_{V}(f),\mathfrak{\mathscr{D}}(g)]_{V})\\
 & +B_{V}(\alpha_{V}(f),[\mathfrak{\mathscr{D}}(g),\alpha_{V}(h)]_{V})+B_{V}(\alpha_{V}(f),[\alpha_{V}(g),\mathfrak{\mathscr{D}}(h)]_{V})\\
 & +B_{V}(\alpha_{V}(g),[\mathfrak{\mathscr{D}}(h),\alpha_{V}(f)]_{V})+B_{V}(\alpha_{V}(g),[\alpha_{V}(h),\mathfrak{\mathscr{D}}(f)]_{V}))e\\
= & 2(B_{V}([\alpha_{V}(h),,\alpha_{V}(f)]_{V},\mathscr{D}(g))+B_{V}([\alpha_{V}(f),,\alpha_{V}(g)]_{V},\mathscr{D}(h))\\
 & +B_{V}([\alpha_{V}(g),,\alpha_{V}(h)]_{V},\mathscr{D}(f)))e\\
= & 0.
\end{align*}

Until now, the Hom-Jacobi identity has been proved to hold.

Moreover, for $f_{1}=j_{1}e^{*}+v_{1}+k_{1}e,\:f_{2}=j_{2}e^{*}+v_{2}+k_{2}e\in L$,
we have
\begin{align*}
\alpha([f_{1},f_{2}])= & \alpha([j_{1}e^{*}+v_{1}+k_{1}e,j_{2}e^{*}+v_{2}+k_{2}e])\\
= & \alpha(j_{1}[e^{*},v_{2}]+j_{2}[v_{1},e^{*}]+[v_{1},v_{2}])\\
= & \alpha(j_{1}\mathscr{D}(v_{2})+j_{2}\mathscr{D}(v_{1})+[v_{1},v_{2}]_{V}+B_{V}(\mathscr{D}(v_{1}),v_{2})e)\\
= & j_{1}\alpha(\mathscr{D}(v_{2}))+j_{2}\alpha(\mathscr{D}(v_{1}))+\alpha([v_{1},v_{2}]_{V})+B_{V}(\mathscr{D}(v_{1}),v_{2})\alpha(e)\\
= & j_{1}\alpha_{V}(\mathscr{D}(v_{2}))+j_{1}B_{V}(x_{0},\mathscr{D}(v_{2}))e+j_{2}\alpha_{V}(\mathscr{D}(v_{1}))+j_{2}B_{V}(x_{0},\mathscr{D}(v_{1}))e\\
 & +\alpha_{V}([v_{1},v_{2}]_{V})+B_{V}(x_{0},[v_{1},v_{2}]_{V})e+B_{V}(\mathscr{D}(v_{1}),v_{2})\lambda e\\
= & j_{1}\alpha_{V}(\mathscr{D}(v_{2}))+j_{2}\alpha_{V}(\mathscr{D}(v_{1}))+j_{1}B_{V}(\mathscr{D}(x_{0}),v_{2})e+j_{2}B_{V}(\mathscr{D}(x_{0}),v_{1})e\\
 & +\alpha_{V}([v_{1},v_{2}]_{V})+B_{V}(x_{0},[v_{1},v_{2}]_{V})e+\lambda B_{V}(\mathscr{D}(v_{1}),v_{2})e
\end{align*}
and
\begin{align*}
 & [\alpha(f_{1}),\alpha(f_{2})]\\
= & [j_{1}\lambda e^{*}+j_{1}x_{0}+\alpha_{V}(v_{1}),j_{2}\lambda e^{*}+j_{2}x_{0}+\alpha_{V}(v_{2})]\\
= & j_{1}j_{2}\lambda[e^{*},x_{0}]+j_{1}\lambda[e^{*},\alpha_{V}(v_{2})]+j_{1}j_{2}\lambda[x_{0},e^{*}]+j_{1}[x_{0},\alpha_{V}(v_{2})]\\
 & +j_{2}\lambda[\alpha_{V}(v_{1}),e^{*}]+j_{2}[\alpha_{V}(v_{1}),x_{0}]+[\alpha_{V}(v_{1}),\alpha_{V}(v_{2})]\\
= & j_{1}\lambda\mathscr{D}(\alpha_{V}(v_{2}))+j_{2}\lambda\mathscr{D}(\alpha_{V}(v_{1}))+j_{1}[x_{0},\alpha_{V}(v_{2})]_{V}\\
 & +j_{1}B_{V}(\mathscr{D}(x_{0}),\alpha_{V}(v_{2}))e+j_{2}[x_{0},\alpha_{V}(v_{1})]_{V}+j_{2}B_{V}(\mathscr{D}(x_{0}),\alpha_{V}(v_{1}))e\\
 & +[\alpha_{V}(v_{1}),\alpha_{V}(v_{2})]_{V}+B_{V}(\mathscr{D}(\alpha_{V}(v_{1})),\alpha_{V}(v_{2}))e\\
= & j_{1}(\lambda\mathscr{D}+\mathrm{ad}_{V}(x_{0}))(\alpha_{V}(v_{2}))+j_{2}(\lambda\mathscr{D}+\mathrm{ad}_{V}(x_{0}))(\alpha_{V}(v_{1}))\\
 & +j_{1}B_{V}(\alpha_{V}(\mathscr{D}(x_{0})),v_{2})e+j_{2}B_{V}(\alpha_{V}(\mathscr{D}(x_{0})),v_{1})e\\
 & +\alpha_{V}([v_{1},v_{2}]_{V})+B_{V}(\alpha_{V}\circ\mathscr{D}\circ\alpha_{V}(v_{1})),v_{2})e\\
= & j_{1}\alpha_{V}\circ\mathscr{D}\circ\alpha_{V}(\alpha_{V}(v_{2}))+j_{2}\alpha_{V}\circ\mathscr{D}\circ\alpha_{V}(\alpha_{V}(v_{1}))\\
 & +j_{1}B_{V}(\mathscr{D}(x_{0}),v_{2})e+j_{2}B_{V}(\mathscr{D}(x_{0}),v_{1})e\\
 & +\alpha_{V}([v_{1},v_{2}]_{V})+B_{V}(\lambda\mathscr{D}+\mathrm{ad}_{V}(x_{0})(v_{1})),v_{2})e\\
= & j_{1}\alpha_{V}\circ\mathscr{D}(\alpha_{V}^{2}(v_{2}))+j_{2}\alpha_{V}\circ\mathscr{D}(\alpha_{V}^{2}(v_{1}))+j_{1}B_{V}(\mathscr{D}(x_{0}),v_{2})e\\
 & +j_{2}B_{V}(\mathscr{D}(x_{0}),v_{1})e+\alpha_{V}([v_{1},v_{2}]_{V})+B_{V}([x_{0},v_{1}]_{V},v_{2})e+\lambda B_{V}(\mathscr{D}(v_{1})),v_{2})e\\
= & j_{1}\alpha_{V}(\mathscr{D}(v_{2}))+j_{2}\alpha_{V}(\mathscr{D}(v_{1}))+j_{1}B_{V}(\mathscr{D}(x_{0}),v_{2})e+j_{2}B_{V}(\mathscr{D}(x_{0}),v_{1})e\\
 & +\alpha_{V}([v_{1},v_{2}]_{V})+B_{V}(x_{0},[v_{1},v_{2}]_{V})e+\lambda B_{V}(\mathscr{D}(v_{1}),v_{2})e.
\end{align*}

It follows that $\alpha([f_{1},f_{2}])=[\alpha(f_{1}),\alpha(f_{2})]$.
Therefore, $(L,[\cdot,\cdot],\alpha)$ is a multiplicative Hom-Lie
algebra.

Finally, the bilinear form $B:L\times L\rightarrow\mathbb{F}$ will
be proved to satisfy Eq. (2.1).

For any $f_{1}=j_{1}e^{*}+v_{1}+k_{1}e,\:f_{2}=j_{2}e^{*}+v_{2}+k_{2}e\in L$,
we have
\[
\alpha(f_{1})=\alpha(j_{1}e^{*}+v_{1}+k_{1}e)=j_{1}\lambda e^{*}+j_{1}x_{0}+(j_{1}\lambda_{0}+k_{1}\lambda+B_{V}(x_{0},v_{1}))e+\alpha_{V}(v_{1})
\]
and
\[
\alpha(f_{2})=\alpha(j_{2}e^{*}+v_{2}+k_{2}e)=j_{2}\lambda e^{*}+j_{2}x_{0}+(j_{2}\lambda_{0}+k_{2}\lambda+B_{V}(x_{0},v_{2}))e+\alpha_{V}(v_{2}).
\]

It follows that
\begin{align*}
 & B(\alpha(f_{1}),f_{2})\\
= & B(j_{1}\lambda e^{*}+j_{1}x_{0}+(j_{1}\lambda_{0}+k_{1}\lambda+B_{V}(x_{0},v_{1}))e+\alpha_{V}(v_{1}),j_{2}e^{*}+v_{2}+k_{2}e)\\
= & j_{1}j_{2}\lambda_{0}+j_{1}j_{2}\lambda B(e^{*},e^{*})+(j_{1}k_{2}+j_{2}k_{1})\lambda+j_{1}B_{V}(x_{0},v_{2})\\
 & +j_{2}B_{V}(x_{0},v_{1})+B_{V}(\alpha_{V}(v_{1}),v_{2})
\end{align*}
and
\begin{align*}
 & B(f_{1},\alpha(f_{2}))\\
= & B(j_{1}e^{*}+v_{1}+k_{1}e,j_{2}\lambda e^{*}+j_{2}x_{0}+(j_{2}\lambda_{0}+k_{2}\lambda+B_{V}(x_{0},v_{2}))e+\alpha_{V}(v_{2}))\\
= & j_{1}j_{2}\lambda_{0}+j_{1}j_{2}\lambda B(e^{*},e^{*})+(j_{1}k_{2}+j_{2}k_{1})\lambda+j_{1}B_{V}(x_{0},v_{2})\\
 & +j_{2}B_{V}(x_{0},v_{1})+B_{V}(v_{1},\alpha_{V}(v_{2})).
\end{align*}

Since $(V,[\cdot,\cdot]_{V},\alpha_{V},B_{V})$ is a quadratic Hom-Lie
algebra, we have $B_{V}(\alpha_{V}(v_{1}),v_{2})=B_{V}(v_{1},\alpha_{V}(v_{2}))$.
Therefore, $B(\alpha(f_{1}),f_{2})=B(f_{1},\alpha(f_{2}))$.

To sum up, the quadruple $(L,[\cdot,\cdot],\alpha,B)$ is a multiplicative
quadratic Hom-Lie algebra.

(2) Let $\alpha$ be invertible. It follows from $\alpha(e)=\lambda e$
that $\lambda\neq0$ .

Conversely, suppose that $\lambda\neq0$. Let $f=je^{*}+v+ke\in L$
such that $\alpha(f)=0$, then
\begin{align*}
0=\alpha(f) & =j(\lambda e^{*}+x_{0}+\lambda_{0}e)+(\alpha_{V}(v)+B_{V}(x_{0},v)e)+k\lambda e\\
 & =(j\lambda)e^{*}+(\alpha_{V}(v)+jx_{0})+(j\lambda_{0}+B_{V}(x_{0},v)+k\lambda)e.
\end{align*}

It follows that $j\lambda=0,\alpha_{V}(v)=jx_{0},j\lambda_{0}+B_{V}(x_{0},v)+k\lambda=0$.
Since $\lambda\neq0$, we obtain $j=0,\alpha_{V}(v)=0,B_{V}(x_{0},v)=k\lambda$.
Therefore, $j=k=0,v=0$ and $f=0$, which implies that $\alpha$ is
invertible.

Moreover, if $\alpha$ is an involution, it follows from $\alpha^{2}(e^{*})=e^{*}$
that
\[
\lambda^{2}e^{*}+(\alpha_{V}(x_{0})+\lambda x_{0})+B_{V}(x_{0},x_{0})e=e^{*}.
\]

Therefore, $\lambda^{2}=1,\alpha_{V}(x_{0})=\lambda x_{0},B_{V}(x_{0},x_{0})=0$.

Conversely, if the condition (3.4) holds, we have $\alpha^{2}(e)=\lambda^{2}e=e$,
\[
\alpha^{2}(e^{*})=\lambda^{2}e^{*}+(\alpha_{V}(x_{0})+\lambda x_{0})+B_{V}(x_{0},x_{0})e=e^{*}+2\alpha_{V}(x_{0})=e^{*}
\]
and
\begin{align*}
\alpha^{2}(v) & =\alpha(\alpha_{V}(v)+B_{V}(x_{0},v)e)=\alpha_{V}^{2}(v)+B_{V}(x_{0},\alpha_{V}(v))e+B_{V}(x_{0},v)\lambda e\\
 & =v+B_{V}(\alpha_{V}(x_{0}),v)e+B_{V}(\lambda x_{0},v)e=v+B_{V}(\alpha_{V}(x_{0})+\lambda x_{0},v)e=v.
\end{align*}

Therefore, for any $f=je^{*}+v+ke\in L$, we have $\alpha^{2}(f)=f$,
that is, $\alpha$ is an involution.
\end{proof}
\begin{remark} If $(V,[\cdot,\cdot]_{V},\alpha_{V},B_{V})$ is an
involutive quadratic Hom-Lie algebra and $\lambda=1$, then the conditions
(3.1)-(3.3) reduce to $x_{0}\in\mathfrak{z}(V)$ such that $\mathscr{D}(x_{0})=0$.

\end{remark}

\begin{definition} The multiplicative quadratic Hom-Lie algebra $(L,[\cdot,\cdot],\alpha,B)$
constructed in the above theorem is said to be the double extension
of the involutive quadratic Hom-Lie algebra $(V,[\cdot,\cdot]_{V},\alpha_{V},B_{V})$
by means of $(\mathscr{D},x_{0},\lambda,\lambda_{0})$.

\end{definition}

If the involutive quadratic Hom-Lie algebra $(V,[\cdot,\cdot]_{V},\alpha_{V},B_{V})$
is restricted, so is $(L,[\cdot,\cdot],\alpha,$ $B)$. More precisely,
the $2$-structure $[2]_{V}$ on $(V,[\cdot,\cdot]_{V},\alpha_{V},B_{V})$
can be extended to its double extension $(L,[\cdot,\cdot],\alpha,B)$.

\begin{theorem} Let $(V,[\cdot,\cdot]_{V},\alpha_{V},B_{V})$ be
a restricted involutive quadratic Hom-Lie algebra and $B_{V}$ $\mathscr{D}$-invariant,
where $\mathscr{D}\in\mathrm{Der}^{p}(V)$ and $\mathscr{D}$ has
$2$-property. If $\lambda=1$ in Theorem 3.1, then for any $m,l\in\mathbb{F}$,
the 2-structure $[2]_{V}$ on $V$ can be extended to its double extension
$L:=\mathscr{E}^{*}\oplus V\oplus\mathscr{E}$ as follows (for any
$u\in V$):
\begin{align*}
u^{[2]_{L}} & =u^{[2]_{V}}+\mathscr{P}(u)e,\\
(e^{*})^{[2]_{L}} & =a_{0}+le+\xi e^{*},\\
e^{[2]_{L}} & =me+u_{0},
\end{align*}
where $a_{0}$ and $\xi$ are as in (2.2) (i.e. the $2$-property),
$u_{0}\in\mathfrak{z}(\alpha_{V}(V))$ such that $\mathscr{D}(u_{0})=0$
and $\mathscr{P}$ is a map satisfying (for any $u,v\in V$and any
$\gamma\in\mathbb{F}$):
\begin{align}
\mathscr{P}(\gamma u) & =\gamma^{2}\mathscr{P}(u),\\
\mathscr{P}(u+v) & =\mathscr{P}(u)+\mathscr{P}(v)+B_{V}(\mathscr{D}(u),v).
\end{align}

\end{theorem}
\begin{proof}
According to Lemma 2.10 for $p=2$, it's suffcient to prove that
\begin{align*}
\mathrm{ad}(\alpha(e))\circ\mathrm{ad}(e) & =\mathrm{ad}(me+u_{0})\circ\alpha,\\
\mathrm{ad}(\alpha(e^{*}))\circ\mathrm{ad}(e^{*}) & =\mathrm{ad}(a_{0}+le+\xi e^{*})\circ\alpha
\end{align*}
and
\[
\mathrm{ad}(\alpha(u))\circ\mathrm{ad}(u)=\mathrm{ad}(u^{[2]_{V}}+\mathscr{P}(u)e)\circ\alpha.
\]

Indeed, for any $f=je^{*}+v+ke\in L$, we have
\[
\alpha(f)=j\alpha(e^{*})+\alpha(v)+k\alpha(e)=je^{*}+(jx_{0}+\alpha_{V}(v))+(j\lambda_{0}+B_{V}(x_{0},v)+k)e.
\]

Therefore,
\begin{align*}
 & \mathrm{ad}(me+u_{0})\circ\alpha(f)+\mathrm{ad}(\alpha(e))\circ\mathrm{ad}(e)(f)\\
= & [me+u_{0},\alpha(f)]+[\alpha(e),[e,f]]\\
= & [u_{0},je^{*}+jx_{0}+\alpha_{V}(v)]\\
= & j[u_{0},e^{*}]+j[u_{0},x_{0}]+[u_{0},\alpha_{V}(v)]\\
= & j\mathscr{D}(u_{0})+j[u_{0},x_{0}]_{V}+jB_{V}(\mathscr{D}(u_{0}),x_{0})e\\
 & +[u_{0},\alpha_{V}(v)]_{V}+B_{V}(\mathscr{D}(u_{0}),\alpha_{V}(v))e\\
= & 0
\end{align*}
and
\begin{align*}
 & \mathrm{ad}(a_{0}+le+\xi e^{*})\circ\alpha(f)+\mathrm{ad}(\alpha(e^{*}))\circ\mathrm{ad}(e^{*})(f)\\
= & [a_{0}+le+\xi e^{*},\alpha(f)]+[\alpha(e^{*}),[e^{*},f]]\\
= & [a_{0}+\xi e^{*},je^{*}+jx_{0}+\alpha_{V}(v)]+[e^{*}+x_{0}+\lambda_{0}e,[e^{*},je^{*}+v+ke]]\\
= & j[a_{0},e^{*}]+j[a_{0},x_{0}]+[a_{0},\alpha_{V}(v)]+\xi[e^{*},\alpha_{V}(v)]+[e^{*}+x_{0},\mathscr{D}(v)]\\
= & j\mathscr{D}(a_{0})+j[a_{0},x_{0}]_{V}+jB_{V}(D(a_{0}),x_{0})e+[a_{0},\alpha_{V}(v)]_{V}\\
 & +B_{V}(\mathscr{D}(a_{0}),\alpha_{V}(v))e+\xi\mathscr{D}(\alpha_{V}(v))+\mathscr{D}^{2}(v)\\
= & \mathrm{ad}_{V}(a_{0})(\alpha_{V}(v))+\xi\mathscr{D}(\alpha_{V}(v))+\mathscr{D}^{2}(v)\\
= & 2\mathscr{D}^{2}(v)\\
= & 0.
\end{align*}

Moreover,
\begin{align*}
 & \mathrm{ad}(u^{[2]_{V}}+\mathscr{P}(u)e)\circ\alpha(f)+\mathrm{ad}(\alpha(u))\circ\mathrm{ad}(u)(f)\\
= & [u^{[2]_{V}}+\mathscr{P}(u)e,\alpha(f)]+[\alpha(u),[u,f]]\\
= & [u^{[2]_{V}},je^{*}+jx_{0}+\alpha_{V}(v)]+[\alpha_{V}(u)+B_{V}(x_{0},u)e,[u,je^{*}+v+ke]]\\
= & j[u^{[2]_{V}},e^{*}]+j[u^{[2]_{V}},x_{0}]+[u^{[2]_{V}},\alpha_{V}(v)]+[\alpha_{V}(u),j[u,e^{*}]+[u,v]]\\
= & j\mathscr{D}(u^{[2]_{V}})+j[x_{0},u^{[2]_{V}}]_{V}+jB_{V}(\mathscr{D}(x_{0}),u^{[2]_{V}})e+[\alpha_{V}(v),u^{[2]_{V}}]_{V}\\
 & +B_{V}(\mathscr{D}(\alpha_{V}(v)),u^{[2]_{V}})e+j[\alpha_{V}(u),\mathscr{D}(u)]+[\alpha_{V}(u),[u,v]]\\
= & j\mathscr{D}(u^{[2]_{V}})+[u^{[2]_{V}},\alpha_{V}(v)]_{V}+B_{V}(\mathscr{D}(u^{[2]_{V}}),\alpha_{V}(v))e+j[\alpha_{V}(u),\mathscr{D}(u)]_{V}\\
 & +jB_{V}(\mathscr{D}(\alpha_{V}(u)),\mathscr{D}(u))e+[\alpha_{V}(u),[u,v]_{V}]_{V}+B_{V}(\mathscr{D}(\alpha_{V}(u)),[u,v]_{V})e\\
= & B_{V}(\mathscr{D}(u^{[2]_{V}}),\alpha_{V}(v))e+jB_{V}(\mathscr{D}(\alpha_{V}(u)),\mathscr{D}(u))e+B_{V}(\mathscr{D}(\alpha_{V}(u)),[u,v]_{V})e.\\
\end{align*}

Since
\begin{align*}
 & B_{V}(\mathscr{D}(\alpha_{V}(u)),\mathscr{D}(u))\\
= & B_{V}(\alpha_{V}(u),\mathscr{D}^{2}(u))\\
= & B_{V}(\alpha_{V}(u),(\xi\mathscr{D}\circ\alpha_{V}+\mathrm{ad}_{V}(a_{0})\circ\alpha_{V})(u))\\
= & B_{V}(\alpha_{V}(u),\xi\mathscr{D}(\alpha_{V}(u))+\mathrm{ad}_{V}(a_{0})(\alpha_{V}(u)))\\
= & \xi B_{V}(\mathscr{D}(\alpha_{V}(u)),\alpha_{V}(u))+B_{V}(\alpha_{V}(u),[a_{0},\alpha_{V}(u)]_{V})\\
= & B_{V}([\alpha_{V}(u),\alpha_{V}(u)]_{V},a_{0})\\
= & 0
\end{align*}
and
\begin{align*}
 & B_{V}(\mathscr{D}(\alpha_{V}(u)),[u,v]_{V})e\\
= & B_{V}(\alpha_{V}(u),\mathscr{D}[u,v]_{V})e\\
= & B_{V}(\alpha_{V}(u),[\mathscr{D}(u),\alpha_{V}(v)]_{V}+[\alpha_{V}(u),\mathscr{D}(v)]_{V})e\\
= & B_{V}(\alpha_{V}(u),[\mathscr{D}(u),\alpha_{V}(v)]_{V})e+B_{V}(\alpha_{V}(u),[\alpha_{V}(u),\mathscr{D}(v)]_{V})e\\
= & B_{V}([\alpha_{V}(u),\mathscr{D}(u)]_{V},\alpha_{V}(v))e\\
= & B_{V}(\mathscr{D}(u^{[2]_{V}}),\alpha_{V}(v))e,
\end{align*}
\\
 we obtain
\begin{align*}
\mathrm{ad}(u^{[2]_{V}}+\mathscr{P}(u)e)\circ\alpha(f)+\mathrm{ad}(\alpha(u))\circ\mathrm{ad}(u)(f) & =2B_{V}(\mathscr{D}(u^{[2]_{V}}),\alpha_{V}(v))e=0.
\end{align*}

Finally, for $u,v\in V$, we have
\begin{align*}
(u+v)^{[2]_{\mathfrak{g}}}= & (u+v)^{[2]_{V}}+\mathscr{P}(u+v)e=u^{[2]_{V}}+v{}^{[2]_{V}}+[u,v]_{V}+\mathscr{P}(u+v)e\\
= & u^{[2]_{\mathfrak{g}}}+\mathscr{P}(u)e+v{}^{[2]_{\mathfrak{g}}}+\mathscr{P}(v)e+[u,v]+B_{V}(\mathscr{D}(u),v)e+\mathscr{P}(u+v)e\\
= & u^{[2]_{\mathfrak{g}}}+v{}^{[2]_{\mathfrak{g}}}+[u,v].
\end{align*}

To sum up, $[2]_{\mathfrak{g}}$ is a $2$-structure on $L$.
\end{proof}
The following theorem is the converse of Theorem 3.4.

\begin{theorem} Let $(L,[\cdot,\cdot],\alpha,B)$ be an irreducible
restricted quadratic Hom-Lie algebra such that $\mathrm{dim}L>1$
and $\mathfrak{z}(L)\neq\{0\}$. If $0\neq e\in\mathfrak{z}(L)$ and
$\mathscr{E}=\mathrm{Span}\{e\}$ such that $\mathscr{E}^{\perp}$
is a $2$-ideal, then $(L,[\cdot,\cdot],\alpha,B)$ is the double
extension of a restricted involutive multiplicative quadratic Hom-Lie
algebra $(V,[\cdot,\cdot]_{V},\alpha_{V},B_{V})$ by means of $\mathscr{D}$,
where $\mathscr{D}\in\mathrm{Der}^{p}(V)$ with $2$-property.

\end{theorem}
\begin{proof}
According to Lemma 2.8, we have $\alpha(\mathfrak{z}(L))\subset\mathfrak{z}(L)$.
Since $0\neq e\in\mathfrak{z}(L)$, the vecter space $\mathscr{E}$
is an ideal of $\mathfrak{g}$. Then there exists $\lambda\in\mathbb{F}$
such that $\alpha(e)=\lambda e$. Moreover, since $\mathscr{E}$ is
an ideal of $L$, $L$ is irreducible and $\mathrm{dim}L>1$, we have
$B(e,e)=0$. It follows from $B$ being nondegenerate that there exists
$0\neq e^{*}\in L$ such that $B(e^{*},e)=1$.

Let $\mathscr{E}^{*}=\mathrm{Span}\{e^{*}\}$ and $\Gamma=\mathscr{E}\oplus\mathscr{E}^{*}$,
then $\Gamma$ is nondegenerate (i.e. $B|_{\Gamma\times\Gamma}$ is
nondegenerate) and $L=\Gamma\oplus\Gamma^{\bot}$. Furthermore, set
$V:=\Gamma^{\bot}$, then $\mathscr{E}{}^{\perp}=\mathscr{E}\oplus V$,
$B|_{V\times V}$ is nondegenerate and there exists a decomposition
$L=\mathscr{E}{}^{*}\oplus V\oplus\mathscr{E}$.

It follows from $\mathscr{E}{}^{\perp}$ being an ideal that $\alpha(\mathscr{E}{}^{\perp})\subset\mathscr{E}{}^{\perp}$.
Then there exists a linear map $\alpha_{V}:V\rightarrow V$ and a
linear form $\kappa:V\rightarrow\mathbb{F}$ such that $\alpha(u)=\alpha_{V}(u)+\kappa(u)e,\forall u\in V$.
In addition, there exists $l\in\mathbb{F},x_{0}\in V$ and $\lambda_{0}\in\mathbb{F}$
such that $\alpha(e^{*})=le^{*}+x_{0}+\lambda_{0}e$.

Be analogue to Theorem 6.6 in \citep{BM}, it can be proved that there
exists a multiplicative quadratic Hom-Lie algebra structure on $V$
such that $L$ is the double extension of $V$ as given in Theorem
3.1.

Denote by $B_{V}:=B|_{V\times V}$ the nondegenerate invariant bilinear
form on $V$. We will prove that there exists a $2$-structure on
$V$.

Let $[2]_{L}:L\rightarrow L,g\mapsto g^{[2]_{L}}$ be a 2-structure
on $L$. Since $V\subset\mathscr{E}{}^{\perp}$, we have
\[
v^{[2]_{L}}\in\mathscr{E}_{p}{}^{\perp}=\mathscr{E}{}^{\perp}=\mathscr{E}\oplus V,\:\forall v\in V.
\]

It follows that $v^{[2]_{L}}=\mathscr{P}(v)e+s(v)$, where $s(v)\in V$.
The map $s:V\rightarrow V,v\mapsto s(v)$ will be proved to be a $2$-structure
on $V$.

Since $(\gamma v)^{[2]_{L}}=\gamma^{2}v{}^{[2]_{L}}$, we have $\mathscr{P}(\gamma v)=\gamma^{2}\mathscr{P}(v)$
and $s(\gamma v)=\gamma^{2}s(v)$.

Besides, for every $u\in V$, we have
\begin{align*}
0= & [v^{[2]_{L}},\alpha(u)]+[\alpha(v),[v,u]]\\
= & [\mathscr{P}(v)e+s(v),\alpha_{V}(u)+B_{V}(x_{0},u)e]\\
 & +[\alpha_{V}(v)+B_{V}(x_{0},v)e,[v,u]_{V}+B_{V}(\mathscr{D}(v),u)e]\\
= & [s(v),\alpha_{V}(u)]+[\alpha_{V}(v),[v,u]_{V}]\\
= & [s(v),\alpha_{V}(u)]_{V}+B_{V}(\mathscr{D}(s(v)),\alpha_{V}(u))e\\
 & +[\alpha_{V}(v),[v,u]_{V}]_{V}+B_{V}(\mathscr{D}(\alpha_{V}(v)),[v,u]_{V})e\\
= & [s(v),\alpha_{V}(u)]_{V}+B_{V}(\mathscr{D}(s(v)),\alpha_{V}(u))e\\
 & +[\alpha_{V}(v),[v,u]_{V}]_{V}+B_{V}([\alpha_{V}(v),\mathscr{D}(v)]_{V},\alpha_{V}(u))e.
\end{align*}

It follows from $B_{V}$ being nondegenerate that
\[
[s(v),\alpha_{V}(u)]_{V}=[\alpha_{V}(v),[v,u]_{V}]_{V}
\]
and
\[
\mathscr{D}(s(v))=[\alpha_{V}(v),\mathscr{D}(v)]_{V}.
\]

Moreover, for any $u,v\in V$, we have
\begin{align*}
0= & (u+v)^{[2]_{L}}+u^{[2]_{L}}+v^{[2]_{L}}+[u,v]\\
= & (\mathscr{P}(u+v)e+s(u+v))+(\mathscr{P}(u)e+s(u))\\
 & +(\mathscr{P}(v)e+s(v))+([u,v]_{V}+B_{V}(\mathscr{D}(u),v)e)\\
= & \text{(}s(u+v)+s(u)+s(v)+[u,v]_{V})\\
 & +(\mathscr{P}(u+v)+\mathscr{P}(u)+\mathscr{P}(v)+B_{V}(\mathscr{D}(u),v))e.
\end{align*}

It follows that $s(u+v)=s(u)+s(v)+[u,v]_{V}$ and $\mathscr{P}(u+v)=\mathscr{P}(u)+\mathscr{P}(v)+B_{V}(\mathscr{D}(u),v)$.

It implies that the map $s$ defines a $2$-structure on $V$, $\mathscr{D}$
is a restricted derivation with respect to $s$ on $V$ and the map
$\mathscr{P}$ satisfies Eqs. (3.5) and (3.6).

Assume that $e^{[2]_{L}}=u_{0}+me+\delta e^{*}$, where $m,\delta\in\mathbb{F}$.
For every $v\in V$, we have
\begin{align*}
0 & =[e^{[2]_{L}},\alpha(v)]+[\alpha(e),[e,v]]\\
 & =[u_{0}+me+\delta e^{*},\alpha_{V}(v)+B_{V}(x_{0},v)e]\\
 & =[u_{0}+\delta e^{*},\alpha_{V}(v)]\\
 & =[u_{0},\alpha_{V}(v)]_{V}+B_{V}(\mathscr{D}(u_{0}),\alpha_{V}(v))e+\delta\mathscr{D}(\alpha_{V}(v))\\
 & =([u_{0},\alpha_{V}(v)]_{V}+\delta\mathscr{D}(\alpha_{V}(v)))+B_{V}(\mathscr{D}(u_{0}),\alpha_{V}(v))e.
\end{align*}

Since $B_{V}$ is nondegenerate, we obtain $\mathscr{D}(u_{0})=0$
and $[u_{0},\alpha_{V}(v)]_{V}+\delta\mathscr{D}(\alpha_{V}(v))=0$.
Therefore, $u_{0}\in\mathfrak{z}(\alpha_{V}(V))$ and $\delta=0$.
It follows that $e^{[2]_{L}}=u_{0}+me$.

Assume now that $(e^{*})^{[2]_{L}}=a_{0}+le+\xi e^{*}$, where $a_{0}\in V,l,\xi\in\mathbb{F}$.
For every $v\in V$, we have
\begin{align*}
0= & [(e^{*})^{[2]_{L}},\alpha(v)]+[\alpha(e^{*}),[e^{*},v]]\\
= & [a_{0}+le+\xi e^{*},\alpha_{V}(v)+B_{V}(x_{0},v)e]+[e^{*}+x_{0}+\lambda_{0}e,\mathscr{D}(v)]\\
= & [a_{0},\alpha_{V}(v)]+\xi[e^{*},\alpha_{V}(v)]+[e^{*},\mathscr{D}(v)]+[x_{0},\mathscr{D}(v)]\\
= & [a_{0},\alpha_{V}(v)]_{V}+B_{V}(\mathscr{D}(a_{0}),\alpha_{V}(v))e+\xi\mathscr{D}(\alpha_{V}(v))\\
 & +\mathscr{D}^{2}(v)+[x_{0},\mathscr{D}(v)]_{V}+B_{V}(\mathscr{D}(x_{0}),\mathscr{D}(v))e\\
= & (\mathscr{D}^{2}+\xi\mathscr{D}\circ\alpha_{V}+\mathrm{ad}_{V}(a_{0})\circ\alpha_{V})(v)+B_{V}(\mathscr{D}(a_{0}),\alpha_{V}(v))e.
\end{align*}

Therefore, $\mathscr{D}^{2}=\xi\mathscr{D}\circ\alpha_{V}+\mathrm{ad}_{V}(a_{0})\circ\alpha_{V}$
and $\mathscr{D}(a_{0})=0$.

The proof is complete.
\end{proof}
In Theorem 3.1, we obtain the double extension of an involutive quadratic
Hom-Lie algebra $(V,[\cdot,\cdot]_{V},\alpha_{V},B_{V})$ by means
of a central extension and its derivation. The following theorem will
give the double extension of an involutive multiplicative quadratic
Hom-Lie algebra by means of an involutive multiplicative Hom-Lie algebra.

To this end, we should first give a lemma as follows:

\begin{lemma} Let $(V,[\cdot,\cdot]_{V},\alpha_{V},B_{V})$ be an
involutive quadratic Hom-Lie algebra and $(A,[\cdot,\cdot]_{A},\alpha_{A})$
an involutive Hom-Lie algebra. Let $\phi:A\rightarrow\mathrm{End}_{a}(V,B_{V})$,
$a\mapsto\phi(a)$ be a representation of $A$ on $(V,[\cdot,\cdot]_{V},\alpha_{V})$
such that
\begin{align}
\phi\circ\alpha_{A}(a) & =\alpha_{V}\circ\phi(a)\circ\alpha_{V},\forall a\in A
\end{align}
and $\psi:V\times V\rightarrow A^{*},(x,y)\mapsto\psi(x,y)$ defined
by $\psi(x,y)(a)=B_{V}(\phi(a)x,y),\forall x,y\in V,\forall a\in A$.
Then, we have $\psi(x_{1},x_{2})\circ\alpha_{A}=\psi(\alpha_{V}(x_{1}),\alpha_{V}(x_{2}))$.

\end{lemma}
\begin{proof}
For every $a\in A$, we have
\[
\psi(x_{1},x_{2})\circ\alpha_{A}(a)=\psi(x_{1},x_{2})(\alpha_{A}(a))=B_{V}(\phi(\alpha_{A}(a))(x_{1}),x_{2})
\]
and
\begin{align*}
\psi(\alpha_{V}(x_{1}),\alpha_{V}(x_{2}))(a) & =B_{V}(\phi(a)(\alpha_{V}(x_{1})),\alpha_{V}(x_{2}))\\
 & =B_{V}(\alpha_{V}\circ\phi(a)(\alpha_{V}(x_{1})),x_{2})\\
 & =B_{V}(\phi(\alpha_{A}(a))(x_{1}),x_{2}).
\end{align*}

Therefore, $\psi(x_{1},x_{2})\circ\alpha_{A}(a)=\psi(\alpha_{V}(x_{1}),\alpha_{V}(x_{2}))(a)$.
Since $a$ is arbitrary, we obtain $\psi(x_{1},x_{2})\circ\alpha_{A}=\psi(\alpha_{V}(x_{1}),\alpha_{V}(x_{2}))$.
\end{proof}
\begin{theorem} (Involutive Double Extension Theorem) Under assumptions
of Lemma 3.6, suppose that $\phi:A\rightarrow\mathrm{End}_{a}(V,B_{V})$
also satisfies
\begin{align}
\alpha_{V}\circ\phi(a)([x,y]_{V}) & =[\phi(a)\circ\alpha_{V}(x),y]_{V}+[x,\phi(a)\circ\alpha_{V}(y)]_{V},\forall x,y\in V.
\end{align}

(1) There exists an involutive multiplicative Hom-Lie algebra structure
on $L=A^{*}\oplus V\oplus A$, where the bracket $[\cdot,\cdot]:L\times L\rightarrow L$
is defined as follows (for any $f,f'\in A^{*},x,x'\in V,a,a'\in A$):
\begin{align*}
[f+x+a,f'+x'+a']= & f\circ\mathrm{ad}_{A}(a')+f'\circ\mathrm{ad}_{A}(a)+\psi(x,x')+[x,x']_{V}\\
 & +\phi(a)(x')+\phi(a')(x)+[a,a']_{A}
\end{align*}
and a linear map $\alpha:L\rightarrow L$ is defined by
\[
\alpha(f+x+a)=f\circ\alpha_{A}+\alpha_{V}(x)+\alpha_{A}(a).
\]

(2) If the bilinear form $B_{\sigma}:L\times L\rightarrow\mathbb{F}$
is defined by
\[
B_{\sigma}(f+x+a,f'+x'+a')=B_{V}(x,x')+f(a')+f'(a)+\sigma(a,a'),
\]
where $\sigma$ is a symmetric nondegenerate invariant bilinear form
on $(A,[\cdot,\cdot]_{A},\alpha_{A})$ such that
\[
\sigma(\alpha_{A}(a),a')=\sigma(a,\alpha_{A}(a')),\:\forall a,a'\in A,
\]
then the quadruple $(L,[\cdot,\cdot],\alpha,B_{\sigma})$ is an involutive
multiplicative quadratic Hom-Lie algebra.

\end{theorem}
\begin{proof}
(1) Based on the fact that $\phi$ is a representation satisfying
Eqs. (3.7) and (3.8), it's easy to prove that $(L,[\cdot,\cdot],\alpha)$
is a Hom-Lie algebra. Next, we will prove that the Hom-Lie algebra
$(L,[\cdot,\cdot],\alpha)$ is multiplicative and involutive.

For $g_{1}=f_{1}+x_{1}+a_{1},\:g_{2}=f_{2}+x_{2}+a_{2}\in L$, we
have
\begin{align*}
[g_{1},g_{2}]= & [f_{1}+x_{1}+a_{1},f_{2}+x_{2}+a_{2}]\\
= & f_{1}\circ\mathrm{ad}_{A}(a_{2})+f_{2}\circ\mathrm{ad}_{A}(a_{1})+\psi(x_{1},x_{2})+[x_{1},x_{2}]_{V}\\
 & +\phi(a_{1})(x_{2})+\phi(a_{2})(x_{1})+[a_{1},a_{2}]_{A}.
\end{align*}

Therefore,
\begin{align*}
\alpha([g_{1},g_{2}])= & f_{1}\circ\mathrm{ad}_{A}(a_{2})\circ\alpha_{A}+f_{2}\circ\mathrm{ad}_{A}(a_{1})\circ\alpha_{A}+\psi(x_{1},x_{2})\circ\alpha_{A}\\
 & +\alpha_{V}([x_{1},x_{2}]_{V})+\alpha_{V}\circ\phi(a_{1})(x_{2})+\alpha_{V}\circ\phi(a_{2})(x_{1})+\alpha_{A}([a_{1},a_{2}]_{A}).
\end{align*}

On the other hand,
\begin{align*}
 & [\alpha(g_{1}),\alpha(g_{2})]\\
= & [\alpha(f_{1}+x_{1}+a_{1}),\alpha(f_{2}+x_{2}+a_{2})]\\
= & [f_{1}\circ\alpha_{A}+\alpha_{V}(x_{1})+\alpha_{A}(a_{1}),f_{2}\circ\alpha_{A}+\alpha_{V}(x_{2})+\alpha_{A}(a_{2})]\\
= & f_{1}\circ\alpha_{A}\circ\mathrm{ad}_{A}(\alpha_{A}(a_{2}))+f_{2}\circ\alpha_{A}\circ\mathrm{ad}_{A}(\alpha_{A}(a_{1}))+\psi(\alpha_{V}(x_{1}),\alpha_{V}(x_{2}))\\
 & +[\alpha_{V}(x_{1}),\alpha_{V}(x_{2})]_{V}+\phi(\alpha_{A}(a_{1}))(\alpha_{V}(x_{2}))+\phi(\alpha_{A}(a_{2}))(\alpha_{V}(x_{1}))+[\alpha_{A}(a_{1}),\alpha_{A}(a_{2})]_{A}\\
= & f_{1}\circ\alpha_{A}\circ\mathrm{ad}_{A}(\alpha_{A}(a_{2}))+f_{2}\circ\alpha_{A}\circ\mathrm{ad}_{A}(\alpha_{A}(a_{1}))+\psi(x_{1},x_{2})\circ\alpha_{A}+\alpha_{V}([x_{1},x_{2}]_{V})\\
 & +\alpha_{V}\circ\phi(a_{1})(\alpha_{V}^{2}(x_{2}))+\alpha_{V}\circ\phi(a_{2})(\alpha_{V}^{2}(x_{1}))+\alpha_{A}([a_{1},a_{2}]_{A})\\
= & f_{1}\circ\alpha_{A}\circ\mathrm{ad}_{A}(\alpha_{A}(a_{2}))+f_{2}\circ\alpha_{A}\circ\mathrm{ad}_{A}(\alpha_{A}(a_{1}))+\psi(x_{1},x_{2})\circ\alpha_{A}+\alpha_{V}([x_{1},x_{2}]_{V})\\
 & +\alpha_{V}\circ\phi(a_{1})(x_{2})+\alpha_{V}\circ\phi(a_{2})(x_{1})+\alpha_{A}([a_{1},a_{2}]_{A}).
\end{align*}

For every $a\in A$, we have
\begin{align*}
 & f_{1}\circ\alpha_{A}\circ\mathrm{ad}_{A}(\alpha_{A}(a_{2}))(a)\\
= & f_{1}\circ\alpha_{A}([\alpha_{A}(a_{2}),a]_{A})=f_{1}([\alpha_{A}^{2}(a_{2}),\alpha_{A}(a)]_{A})\\
= & f_{1}([a_{2},\alpha_{A}(a)]_{A})=f_{1}\circ\mathrm{ad}_{A}(a_{2})\circ\alpha_{A}(a).
\end{align*}

It follows that $f_{1}\circ\alpha_{A}\circ\mathrm{ad}_{A}(\alpha_{A}(a_{2}))=f_{1}\circ\mathrm{ad}_{A}(a_{2})\circ\alpha_{A}$.
Similarly, we have $f_{2}\circ\alpha_{A}\circ\mathrm{ad}_{A}(\alpha_{A}(a_{1}))=f_{2}\circ\mathrm{ad}_{A}(a_{1})\circ\alpha_{A}$.
Therefore, $\alpha([g_{1},g_{2}])=[\alpha(g_{1}),\alpha(g_{2})]$.

Furthermore, for any $g=f+x+a\in L$, we have
\begin{align*}
\alpha^{2}(g)= & \alpha(\alpha(f+x+a))=\alpha(f\circ\alpha_{A}+\alpha_{V}(x)+\alpha_{A}(a))=f\circ\alpha_{A}\circ\alpha_{A}+\alpha_{V}^{2}(x)+\alpha_{A}^{2}(a)\\
= & f\circ(\alpha_{A}^{2})+x+a=f+x+a=g.
\end{align*}

It follows that $\alpha^{2}=\mathrm{id}$. Therefore, the triple $(L,[\cdot,\cdot],\alpha)$
is an involutive multiplicative Hom-Lie algebra.

(2) Clearly, the bilinear form $B_{\sigma}$ is symmetric, nondegenerate
and invariant. Next, we will prove that $B_{\sigma}$ satisfies Eq.
(2.1).

Indeed, for $g_{1}=f_{1}+x_{1}+a_{1},\:g_{2}=f_{2}+x_{2}+a_{2}\in L$,
we have
\begin{align*}
B_{\sigma}(\alpha(g_{1}),g_{2}) & =B_{\sigma}(\alpha(f_{1}+x_{1}+a_{1}),f_{2}+x_{2}+a_{2})\\
 & =B_{\sigma}(f_{1}\circ\alpha_{A}+\alpha_{V}(x_{1})+\alpha_{A}(a_{1}),f_{2}+x_{2}+a_{2})\\
 & =B_{V}(\alpha_{V}(x_{1}),x_{2})+f_{1}\circ\alpha_{A}(a_{2})+f_{2}\circ\alpha_{A}(a_{1})+\sigma(\alpha_{A}(a_{1}),a_{2})
\end{align*}
and
\begin{align*}
B_{\sigma}(g_{1},\alpha(g_{2})) & =B_{\sigma}(f_{1}+x_{1}+a_{1},\alpha(f_{2}+x_{2}+a_{2}))\\
 & =B_{\sigma}(f_{1}+x_{1}+a_{1},f_{2}\circ\alpha_{A}+\alpha_{V}(x_{2})+\alpha_{A}(a_{2}))\\
 & =B_{V}(x_{1},\alpha_{V}(x_{2}))+f_{1}\circ\alpha_{A}(a_{2})+f_{2}\circ\alpha_{A}(a_{1})+\sigma(a_{1},\alpha_{A}(a_{2}))\\
 & =B_{V}(\alpha_{V}(x_{1}),x_{2})+f_{1}\circ\alpha_{A}(a_{2})+f_{2}\circ\alpha_{A}(a_{1})+\sigma(\alpha_{A}(a_{1}),a_{2}).
\end{align*}

Therefore, $B_{\sigma}(\alpha(g_{1}),g_{2})=B_{\sigma}(g_{1},\alpha(g_{2}))$,
which implies that $B_{\sigma}$ is a quadratic structure on $(L,[\cdot,\cdot],\alpha)$
and the quadruple $(L,[\cdot,\cdot],\alpha,B_{\sigma})$ is an involutive
multiplicative quadratic Hom-Lie algebra.
\end{proof}
Next, we will study the isomorphism between double extensions of the
quadratic Hom-Lie algebra $(V,[\cdot,\cdot]_{V},\alpha_{V},B_{V})$.
In Theorem 3.1, a double extension of $V$ by means of $(\mathscr{D},x_{0},\lambda,\lambda_{0})$
is constructed, which is denoted by $L=\mathscr{E}^{*}\oplus V\oplus\mathscr{E}$.
Now, let $\widetilde{L}=\mathscr{\widetilde{E}}^{*}\oplus V\oplus\mathscr{\widetilde{E}}$
be another double extension of $V$ by means of $(\mathscr{\widetilde{D}},\widetilde{x}_{0},\widetilde{\lambda},\widetilde{\lambda}_{0})$,
where $\mathscr{\widetilde{D}}\in\mathrm{Der}_{\alpha_{V}}(V)$, $\mathscr{\widetilde{E}}=\mathrm{Span}\{\widetilde{e}\}$,
$\mathscr{\widetilde{E}}^{*}=\mathrm{Span}\{\widetilde{e}^{*}\}$
for $\widetilde{e}^{*}=\widetilde{\mathscr{D}}$, $\widetilde{x}_{0}\in V$
and $\widetilde{\lambda},\widetilde{\lambda}_{0}\in\mathbb{F}$. An
adapted isomorphism between $L$ and $\widetilde{L}$ is a bijection
$\pi:L\rightarrow\widetilde{L}$ which satisfies (for any $f,g\in L$):
\begin{align}
\pi([f,g]_{\mathfrak{g}}) & =[\pi(f),\pi(g)]_{\widetilde{\mathfrak{g}}},\\
B_{\mathfrak{\widetilde{g}}}(\pi(f),\pi(g)) & =B_{\mathfrak{g}}(f,g),\\
\pi(\alpha(f)) & =\widetilde{\alpha}(\pi(f)),\\
\pi(\mathscr{E}\oplus V) & =\mathscr{\widetilde{E}}\oplus V.
\end{align}

Let $\mathrm{pr}:\mathscr{\widetilde{E}}\oplus V\rightarrow V$ be
a projection and $\pi_{0}:=\mathrm{pr}\circ\pi$. Then $\pi_{0}$
is a linear map on $V$. Be similar to \citep{BeBou}, it's easy to
prove that for any $u,v\in V$,
\begin{equation}
\pi_{0}([u,v]_{V})=[\pi_{0}(u),\pi_{0}(v)]_{V}
\end{equation}
and
\begin{equation}
B_{V}(\pi_{0}(u),\pi_{0}(v))=B_{V}(u,v).
\end{equation}

Moreover, be similar to \citep{BeBou}, according to the conditions
(3.10), (3.11) and (3.13), we obtain the adapted isomorphism $\pi:\mathscr{E}^{*}\oplus V\oplus\mathscr{E}\rightarrow\mathscr{\widetilde{E}}^{*}\oplus V\oplus\mathscr{\widetilde{E}}$
as follows:
\begin{align*}
\pi(u) & =\pi_{0}(u)+B_{V}(t_{\pi},u)\widetilde{e},\:\forall u\in V,\\
\pi(e) & =\gamma\widetilde{e},\\
\pi(e^{*}) & =\gamma^{-1}(\widetilde{e}^{*}+\pi_{0}(t_{\pi}))+\nu\widetilde{e},
\end{align*}
where $t_{\pi}\in V$ (depending only on $\pi$), $\gamma\in\mathbb{F}$
and $\nu\in\mathbb{F}$ is arbitrary.

Based on the conditions (3.9), (3.10) and (3.12), we can also obtain
\[
\pi_{0}^{-1}\mathscr{\widetilde{D}}\pi_{0}=\gamma\mathscr{D}+\mathrm{ad}_{V}(t_{\pi})
\]
and
\[
B_{\mathfrak{g}}(e^{*},e^{*})=B_{V}(t_{\pi},t_{\pi})+\gamma^{-2}B_{\mathfrak{\widetilde{g}}}(\widetilde{e}^{*},\widetilde{e}^{*}).
\]

Next, the condition (3.11) should be considered additionally so that
$\pi:L\rightarrow\widetilde{L}$ is an adapted isomorphism.

It follows from $\pi(\alpha(e^{*}))=\widetilde{\alpha}(\pi(e^{*}))$
that
\begin{align*}
\gamma^{-1}\widetilde{\lambda} & =\gamma^{-1}\lambda,\\
\gamma^{-1}\alpha_{V}(\pi_{0}(t_{\pi}))+\gamma^{-1}\widetilde{x}_{0} & =\gamma^{-1}\lambda\pi_{0}(t_{\pi})+\pi_{0}(x_{0}),
\end{align*}
and
\[
\widetilde{\lambda}\nu+\gamma^{-1}B_{V}(\widetilde{x}_{0},\pi_{0}(t_{\pi}))+\gamma^{-1}\widetilde{\lambda}_{0}=\lambda\nu+B_{V}(x_{0},t_{\pi})+\gamma\lambda_{0}.
\]

Therefore, we obtain

\begin{align*}
\widetilde{\lambda} & =\lambda,\\
\pi_{0}(\alpha_{V}(t_{\pi})+\lambda t_{\pi}) & =\gamma\pi_{0}(x_{0})+\widetilde{x}_{0},
\end{align*}
and
\[
B_{V}(\widetilde{x}_{0},\pi_{0}(t_{\pi}))+\gamma B_{V}(x_{0},t_{\pi})=\widetilde{\lambda}_{0}+\gamma^{2}\lambda_{0}.
\]

Since $\pi(\alpha(v))=\widetilde{\alpha}(\pi(v))$, we have
\begin{equation}
\alpha_{V}(\pi_{0}(v))=\pi_{0}(\alpha_{V}(v))
\end{equation}
and
\[
B_{V}(\widetilde{x}_{0},\pi_{0}(v))+\gamma B_{V}(x_{0},v))=B_{V}(t_{\pi},\alpha_{V}(v)+\lambda v).
\]

Now, Eqs. $(3.13)-(3.15)$ imply that $\pi_{0}$ is an automorphism
on $V$.

We arrive at the following theorem.

\begin{theorem} Let $B_{V}$ be $\mathscr{D}$- and $\mathscr{\widetilde{D}}$-invariant,
where $\mathscr{D},\mathscr{\widetilde{D}}\in\mathrm{Der}_{\alpha_{V}}(V)$.
Then there is an adapted isomorphism $\pi:L\rightarrow\widetilde{L}$
if and only if there exists exactly an automorphism $\pi_{0}:V\rightarrow V$,
a scalar $\gamma\in\mathbb{F}$\textbackslash$\{0\}$ and $t_{\pi}\in V$
(depending only on $\pi$) such that

\begin{align*}
\pi_{0}^{-1}\mathscr{\widetilde{D}}\pi_{0} & =\gamma\mathscr{D}+\mathrm{ad}_{V}(t_{\pi}),\\
B_{\mathfrak{g}}(e^{*},e^{*}) & =B_{V}(t_{\pi},t_{\pi})+\gamma^{-2}B_{\mathfrak{\widetilde{g}}}(\widetilde{e}^{*},\widetilde{e}^{*}),\\
\widetilde{\lambda} & =\lambda,\\
\pi_{0}(\alpha_{V}(t_{\pi})+\lambda t_{\pi}) & =\gamma\pi_{0}(x_{0})+\widetilde{x}_{0},\\
B_{V}(\widetilde{x}_{0},\pi_{0}(t_{\pi}))+\gamma B_{V}(x_{0},t_{\pi}) & =\widetilde{\lambda}_{0}+\gamma^{2}\lambda_{0},\\
B_{V}(\widetilde{x}_{0},\pi_{0}(v))+\gamma B_{V}(x_{0},v) & =B_{V}(t_{\pi},\alpha_{V}(v)+\lambda v)
\end{align*}
and

\begin{align*}
\pi(u) & =\pi_{0}(u)+B_{V}(t_{\pi},u)\widetilde{e},\:\forall u\in V,\\
\pi(e) & =\gamma\widetilde{e},\\
\pi(e^{*}) & =\gamma^{-1}(\widetilde{e}^{*}+\pi_{0}(t_{\pi}))+\nu\widetilde{e},
\end{align*}
where $\nu$ is arbitrary.

\end{theorem}

Let $(V,[\cdot,\cdot]_{V},\alpha_{V},B_{V})$ be a restricted Hom-Lie
algebra and $[2]_{V}$ a $2$-structure on $V$. According to Theorem
3.4, the $2$-structure $[2]_{V}$ can be extended to any double extension
of $V$. For the double extension $L$ (resp. $\widetilde{L}$) of
$V$, $[2]_{L}$ (resp. $[2]_{\widetilde{L}}$) denotes the $2$-structure
on $L$ (resp. $\widetilde{L}$) written in terms of $m,u_{0},a_{0},l,\xi$
and $\mathscr{P}$ (resp. $\widetilde{m},\widetilde{u}_{0},\widetilde{a}_{0},\widetilde{l},\widetilde{\xi}$
and $\mathscr{\widetilde{P}}$). We will study the equivalence class
of $2$-structures on double extensions.

\begin{theorem} The adapted isomorphism $\pi:L\rightarrow\widetilde{L}$
given in Theorem 3.8 is restricted if and only if

\[
\pi_{0}(u^{[2]_{V}})=\pi_{0}(u)^{[2]_{V}}+B_{V}(t_{\pi},u)^{2}\widetilde{u}_{0}
\]
and
\begin{align*}
\mathscr{\widetilde{P}}\circ\pi_{0} & =\gamma\mathscr{P}+B_{V}(t_{\pi},(\cdot)^{[2]_{V}})+B_{V}(t_{\pi},\cdot)^{2}\widetilde{m},\\
\widetilde{a}_{0} & =\gamma^{2}(\pi_{0}(a_{0})+\gamma^{-1}\xi\pi_{0}(t_{\pi}))+\gamma^{2}\nu^{2}\widetilde{u}_{0}+\mathscr{\widetilde{D}}(\pi_{0}(t_{\pi}))+\pi_{0}(t_{\pi})^{[2]_{V}},\\
\widetilde{l} & =\gamma^{2}(B_{V}(t_{\pi},a_{0})+\gamma l+\nu\xi)+\mathscr{\widetilde{P}}(\pi_{0}(t_{\pi}))+\nu^{2}\widetilde{m},\\
\widetilde{\xi} & =\gamma\xi,\\
\widetilde{m} & =\gamma^{-2}(\gamma m+B_{V}(t_{\pi},u_{0}))\\
\widetilde{u}_{0} & =\gamma^{-2}\pi_{0}(u_{0}).
\end{align*}

\end{theorem}
\begin{proof}
($\Leftarrow$) Obviously.

($\Rightarrow$) Assume that $\pi$ is a restricted isomorphism.

For every $u\in V$, we have

\begin{align*}
\pi(u^{[2]_{\mathfrak{g}}}) & =\pi(u^{[2]_{V}}+\mathscr{P}(u)e)=\pi(u^{[2]_{V}})+\mathscr{P}(u)\pi(e)\\
 & =\pi_{0}(u^{[2]_{V}})+B_{V}(t_{\pi},u^{[2]_{V}})\widetilde{e}+\mathscr{P}(u)\gamma\widetilde{e}\\
 & =\pi_{0}(u^{[2]_{V}})+(B_{V}(t_{\pi},u^{[2]_{V}})+\mathscr{P}(u)\gamma\text{)}\widetilde{e}
\end{align*}
and

\begin{align*}
(\pi(u))^{[2]_{\mathfrak{\widetilde{g}}}} & =(\pi_{0}(u)+B_{V}(t_{\pi},u)\widetilde{e})^{[2]_{\mathfrak{\widetilde{g}}}}\\
 & =(\pi_{0}(u))^{[2]_{\mathfrak{\widetilde{g}}}}+B_{V}(t_{\pi},u)^{2}\widetilde{e}{}^{[2]_{\mathfrak{\widetilde{g}}}}+[\pi_{0}(u),B_{V}(t_{\pi},u)\widetilde{e}]\\
 & =(\pi_{0}(u))^{[2]_{V}}+\mathscr{\widetilde{P}}(\pi_{0}(u))\widetilde{e}+B_{V}(t_{\pi},u)^{2}(\widetilde{m}\widetilde{e}+\widetilde{u}_{0})\\
 & =((\pi_{0}(u))^{[2]_{V}}+B_{V}(t_{\pi},u)^{2}\widetilde{u}_{0})+(\mathscr{\widetilde{P}}(\pi_{0}(u))+B_{V}(t_{\pi},u)^{2}\widetilde{m})\widetilde{e}.
\end{align*}

It follows from $\pi(u^{[2]_{\mathfrak{g}}})=(\pi(u))^{[2]_{\mathfrak{\widetilde{g}}}}$
that $\pi_{0}(u^{[2]_{V}})=\pi_{0}(u)^{[2]_{V}}+B_{V}(t_{\pi},u)^{2}\widetilde{u}_{0}$
and $\mathscr{\widetilde{P}}\circ\pi_{0}=\gamma\mathscr{P}+B_{V}(t_{\pi},(\cdot)^{[2]_{V}})+B_{V}(t_{\pi},\cdot)^{2}\widetilde{m}$.

Moreover, we have
\begin{align*}
\pi((e^{*})^{[2]_{\mathfrak{g}}}) & =\pi(a_{0}+le+\xi e^{*})=\pi(a_{0})+l\pi(e)+\xi\pi(e^{*})\\
 & =\pi_{0}(a_{0})+B_{V}(t_{\pi},a_{0})\widetilde{e}+l\gamma\widetilde{e}+\xi(\gamma^{-1}(\widetilde{e}^{*}+\pi_{0}(t_{\pi}))+\nu\widetilde{e})\\
 & =(\pi_{0}(a_{0})+\xi\gamma^{-1}\pi_{0}(t_{\pi}))+(l\gamma+\nu\xi+B_{V}(t_{\pi},a_{0}))\widetilde{e}++\xi\gamma^{-1}\widetilde{e}^{*}
\end{align*}
and
\begin{align*}
(\pi(e^{*}))^{[2]_{\mathfrak{\widetilde{g}}}}= & (\gamma^{-1}(\widetilde{e}^{*}+\pi_{0}(t_{\pi}))+\nu\widetilde{e})^{[2]_{\mathfrak{\widetilde{g}}}}\\
= & \gamma^{-2}(\widetilde{e}^{*}+\pi_{0}(t_{\pi}))^{[2]_{\mathfrak{\widetilde{g}}}}+\nu^{2}\widetilde{e}{}^{[2]_{\mathfrak{\widetilde{g}}}}+[\gamma^{-1}(\widetilde{e}^{*}+\pi_{0}(t_{\pi})),\nu\widetilde{e}]\\
= & \gamma^{-2}(\widetilde{e}^{*})^{[2]_{\mathfrak{\widetilde{g}}}}+\gamma^{-2}\pi_{0}(t_{\pi})^{[2]_{\mathfrak{\widetilde{g}}}}+\gamma^{-2}[\widetilde{e}^{*},\pi_{0}(t_{\pi})]+\nu^{2}\widetilde{e}{}^{[2]_{\mathfrak{\widetilde{g}}}}\\
= & \gamma^{-2}(\widetilde{a}_{0}+\widetilde{l}\widetilde{e}+\widetilde{\xi}\widetilde{e}^{*})+\gamma^{-2}(\pi_{0}(t_{\pi})^{[2]_{V}}+\widetilde{\mathscr{P}}(\pi_{0}(t_{\pi}))\widetilde{e})\\
 & +\gamma^{-2}\mathscr{\widetilde{D}}(\pi_{0}(t_{\pi}))+\nu^{2}\widetilde{m}\widetilde{e}+\nu^{2}\widetilde{u}_{0}\\
= & (\gamma^{-2}\widetilde{a}_{0}+\nu^{2}\widetilde{u}_{0}+\gamma^{-2}\mathscr{\widetilde{D}}(\pi_{0}(t_{\pi}))+\gamma^{-2}\pi_{0}(t_{\pi})^{[2]_{V}})\\
 & +(\gamma^{-2}\widetilde{l}+\gamma^{-2}\widetilde{\mathscr{P}}(\pi_{0}(t_{\pi}))+\nu^{2}\widetilde{m})\widetilde{e}+\gamma^{-2}\widetilde{\xi}\widetilde{e}^{*}.
\end{align*}

Since $\pi((e^{*})^{[2]_{\mathfrak{g}}})=(\pi(e^{*}))^{[2]_{\mathfrak{\widetilde{g}}}}$,
we obtain
\begin{align*}
\widetilde{\xi} & =\gamma\xi,\\
\widetilde{l} & =\gamma^{2}(B_{V}(t_{\pi},a_{0})+\gamma l+\nu\xi)+\mathscr{\widetilde{P}}(\pi_{0}(t_{\pi}))+\nu^{2}\widetilde{m}
\end{align*}
and
\[
\widetilde{a}_{0}=\gamma^{2}(\pi_{0}(a_{0})+\gamma^{-1}\xi\pi_{0}(t_{\pi}))+\gamma^{2}\nu^{2}\widetilde{u}_{0}+\mathscr{\widetilde{D}}(\pi_{0}(t_{\pi}))+\pi_{0}(t_{\pi})^{[2]_{V}}.
\]

Finally, since
\begin{align*}
0 & =\pi(e^{[2]_{\mathfrak{g}}})+(\pi(e))^{[2]_{\mathfrak{\widetilde{g}}}}=\pi(me+u_{0})+(\gamma\widetilde{e})^{[2]_{\mathfrak{\widetilde{g}}}}\\
 & =m\pi(e)+\pi(u_{0})+\gamma^{2}\widetilde{e}{}^{[2]_{\mathfrak{\widetilde{g}}}}=m\gamma\widetilde{e}+\pi_{0}(u_{0})+B_{V}(t_{\pi},u_{0})\widetilde{e}+\gamma^{2}(\widetilde{m}\widetilde{e}+\widetilde{u}_{0})\\
 & =(\widetilde{m}\gamma^{2}+m\gamma+B_{V}(t_{\pi},u_{0}))\widetilde{e}+(\gamma^{2}\widetilde{u}_{0}+\pi_{0}(u_{0})),
\end{align*}
we have $\widetilde{m}=\gamma^{-2}(\gamma m+B_{V}(t_{\pi},u_{0}))$
and $\widetilde{u}_{0}=\gamma^{-2}\pi_{0}(u_{0}).$

The proof is complete.
\end{proof}

\section{Double Extensions of Restricted Hom-Lie Algebras for $p>2$}

In this section, assume that $\mathbb{F}$ is an algebraically closed
field of characteristic $p>2$. Let $(V,[\cdot,\cdot]_{V},\alpha_{V})$
be a Hom-Lie algebra over $\mathbb{F}$. This is the double extension
theorem over $\mathbb{F}$, whose proof is identical to that of characteristic
zero \citep{BM}.

\begin{theorem} (Double Extension Theorem) Let $(V,[\cdot,\cdot]_{V},\alpha_{V},B_{V})$
be an involutive quadratic Hom-Lie algebra, $\mathscr{D}\in\mathrm{Der}_{\alpha_{V}}(V)$
which makes $B_{V}$ $\mathscr{D}$-invariant, $\mathscr{E}=\mathrm{Span}\{e\}$
and $\mathscr{E}^{*}=\mathrm{Span}\{e^{*}\}$, where $e^{*}=\mathscr{D}$.

(1) Suppose that there exist $x_{0}\in V,\lambda\in\mathbb{F}$ such
that
\begin{align}
\lambda\mathscr{D}+\mathrm{ad}_{V}(x_{0}) & =\mathscr{D}\\
\alpha_{V}(\mathscr{D}(x_{0})) & =-\mathscr{D}(x_{0})\\
\alpha_{V}\circ\mathscr{D}^{2}-\mathscr{D}^{2}\circ\alpha_{V} & =\mathrm{ad}_{V}(\mathscr{D}(x_{0})),
\end{align}
where $\mathrm{ad}_{V}(x):=[x,\cdot]_{V}$, for $x\in V$, then there
exists a multiplicative quadratic Hom-Lie algebra structure on $L:=\mathscr{E}^{*}\oplus V\oplus\mathscr{E}$,
where the bracket $[\cdot,\cdot]:L\times L\rightarrow L$ is defined
by
\begin{align*}
[x,y] & =[x,y]_{V}+B_{V}(\mathscr{D}(x),y)e,\:\forall x,y\in V,\\{}
[e^{*},x] & =\mathscr{D}(x),\forall x\in V,\\{}
[e,z] & =0,\forall z\in L
\end{align*}
and the linear map $\alpha:L\rightarrow L$ is defined as follows:
\[
\alpha(x)=\alpha_{V}(x)+B_{V}(x_{0},x)e,\:\forall x\in V,\:\alpha(e^{*})=\lambda e^{*}+x_{0}+\lambda_{0}e,\:\alpha(e)=\lambda e,
\]
where $\lambda_{0}\in\mathbb{F}$.

The symmetric nondegenerate invariant bilinear form $B:L\times L\rightarrow\mathbb{F}$
is defined by
\begin{align*}
B(x,y) & =B_{V}(x,y),\:B(x,e^{*})=B(x,e)=0,\:\:\mathrm{for\:any}\:x,y\in V,\\
B(e^{*},e) & =1,\:B(e^{*},e^{*})=B(e,e)=0.
\end{align*}

(2) The twist map $\alpha$ is invertible if and only if $\lambda\neq0$.
Furthermore, $\alpha$ is an involution if and only if
\begin{equation}
\lambda^{2}=1,\:\alpha_{V}(x_{0})=-\lambda x_{0},\:\lambda_{0}=-\frac{1}{2}B_{V}(x_{0},x_{0}).
\end{equation}

\end{theorem}

\begin{remark} If $(V,[\cdot,\cdot]_{V},\alpha_{V},B_{V})$ is an
involutive quadratic Hom-Lie algebra and $\lambda=1$, then the conditions
(4.1)-(4.3) reduce to $x_{0}\in\mathfrak{z}(V)$ such that $\mathscr{D}(x_{0})=0$.

\end{remark}

For $u,v\in V$, denote by $\eta_{i}^{V}(u,v)$ the coefficient obtained
from the expansion

\[
B_{V}(\mathscr{D}(\alpha_{V}^{p-2}(\lambda u+v)),\mathrm{ad}_{V}(\alpha_{V}^{p-3}(\lambda u+v))\circ\cdot\cdot\cdot\circ\mathrm{ad}_{V}(\lambda u+v)(u))=\sum_{1\leq i\leq p-1}i\eta_{i}^{V}(u,v)\lambda^{i-1}.
\]

\begin{lemma} Under the assumption of Theorem 4.1, we have

\begin{equation}
s_{i}^{L}(u,v)=s_{i}^{V}(u,v)+\eta_{i}^{V}(u,v)e,\:\forall u,v\in V,\:\forall1\leq i\leq p-1.
\end{equation}

\end{lemma}
\begin{proof}
Indeed, for any $\lambda\in\mathbb{F}$ and any $u,v\in V$, we have
\begin{align*}
 & \mathrm{ad}(\alpha^{p-2}(\lambda u+v))\circ\cdot\cdot\cdot\circ\mathrm{ad}(\lambda u+v)(u)-\mathrm{ad}_{V}(\alpha_{V}^{p-2}(\lambda u+v))\circ\cdot\cdot\cdot\circ\mathrm{ad}_{V}(\lambda u+v)(u)\\
= & B_{V}(\mathscr{D}(\alpha_{V}^{p-2}(\lambda u+v)),\mathrm{ad}_{V}(\alpha_{V}^{p-3}(\lambda u+v))\circ\cdot\cdot\cdot\circ\mathrm{ad}_{V}(\lambda u+v)(u))e
\end{align*}

A direct computation leads to Eq. (4.5).
\end{proof}
\begin{theorem} Let $(V,[\cdot,\cdot]_{V},\alpha_{V},B_{V})$ be
a restricted involutive quadratic Hom-Lie algebra and $B_{V}$ $\mathscr{D}$-invariant,
where $\mathscr{D}\in\mathrm{Der}^{p}(V)$ and $\mathscr{D}$ has
$p$-property. If $\lambda=1$ in Theorem 4.1, then the $p$-structure
$[p]_{V}$ on $V$ can be extended to its double extension $L:=\mathscr{E}^{*}\oplus V\oplus\mathscr{E}$
as follows (for every $u\in V$):
\begin{align*}
u^{[p]_{L}} & :=u^{[p]_{V}}+\mathscr{P}(u)e,\\
(e^{*})^{[p]_{L}} & :=a_{0}+le+\xi e^{*},\\
e^{[p]_{L}} & :=me+u_{0},
\end{align*}
where $a_{0}$ and $\xi$ are as in Eq. (2.2) (the $p$-property),
$u_{0}\in\mathfrak{z}(V)$ such that $\mathscr{D}(u_{0})=0$ and the
map $\mathscr{P}$ satisfies (for any $u,v\in V$ and any $\gamma\in\mathbb{F}$):
\begin{equation}
\mathscr{P}(\gamma u)=\gamma^{p}\mathscr{P}(u),\:\mathscr{P}(u+v)=\mathscr{P}(u)+\mathscr{P}(v)+\sum_{1\leq i\leq p-1}\eta_{i}^{V}(u,v).
\end{equation}

\end{theorem}
\begin{proof}
According to Lemma 2.10 for $p>2$, it's sufficient to prove that
\begin{align*}
\mathrm{ad}(\alpha^{p-1}(e))\circ\mathrm{ad}(\alpha^{p-2}(e))\circ\cdot\cdot\cdot\circ\mathrm{ad}(e) & =\mathrm{ad}(me+u_{0})\circ\alpha^{p-1},\\
\mathrm{ad}(\alpha^{p-1}(e^{*}))\circ\mathrm{ad}(\alpha^{p-2}(e^{*}))\circ\cdot\cdot\cdot\circ\mathrm{ad}(e^{*}) & =\mathrm{ad}(a_{0}+le+\xi e^{*})\circ\alpha^{p-1}
\end{align*}
and
\[
\mathrm{ad}(\alpha^{p-1}(u))\circ\mathrm{ad}(\alpha^{p-2}(u))\circ\cdot\cdot\cdot\circ\mathrm{ad}(u)=\mathrm{ad}(u^{[p]_{V}}+\mathscr{P}(u)e)\circ\alpha^{p-1}.
\]

Indeed, for every $f=je^{*}+v+ke\in L$, we have
\begin{align*}
 & \mathrm{ad}(me+u_{0})\circ\alpha^{p-1}(f)-\mathrm{ad}(\alpha^{p-1}(e))\circ\mathrm{ad}(\alpha^{p-2}(e))\circ\cdot\cdot\cdot\circ\mathrm{ad}(e)(f)\\
= & [me+u_{0},je^{*}+j\sum_{1\leq i\leq p-1}\alpha_{V}^{i-1}(x_{0})+\alpha_{V}^{p-1}(v)]\\
= & j[u_{0},e^{*}]+j\sum_{1\leq i\leq p-1}[u_{0},\alpha_{V}^{i-1}(x_{0})]+[u_{0},\alpha_{V}^{p-1}(v)]\\
= & -j\mathscr{D}(u_{0})+j\sum_{1\leq i\leq p-1}([u_{0},\alpha_{V}^{i-1}(x_{0})]_{V}+B_{V}(\mathscr{D}(u_{0}),\alpha_{V}^{i-1}(x_{0}))e)\\
 & +[u_{0},\alpha_{V}^{p-1}(v)]_{V}+B_{V}(\mathscr{D}(u_{0}),\alpha_{V}^{p-1}(v))e\\
= & 0,
\end{align*}
and
\begin{align*}
 & \mathrm{ad}(\alpha^{p-1}(e^{*}))\circ\mathrm{ad}(\alpha^{p-2}(e^{*}))\circ\cdot\cdot\cdot\circ\mathrm{ad}(e^{*})(f)\\
= & [e^{*}+\sum_{0\leq j\leq p-2}\alpha_{V}^{j}(x_{0})+c_{p-1}e,[e^{*}+\sum_{0\leq j\leq p-3}\alpha_{V}^{j}(x_{0})+c_{p-2}e,\cdot\cdot\cdot[e^{*},je^{*}+v+ke]]]\\
= & [e^{*}+\sum_{0\leq j\leq p-2}\alpha_{V}^{j}(x_{0}),[e^{*}+\sum_{0\leq j\leq p-3}\alpha_{V}^{j}(x_{0}),\cdot\cdot\cdot[e^{*}+x_{0},\mathscr{D}(v)]]],\:(c_{i}\in\mathbb{F},\:\forall i\in\mathbb{N}).
\end{align*}

It follows from $x_{0}\in\mathfrak{z}(V)$, $\mathscr{D}(x_{0})=0$,
$\alpha_{V}^{2}=\mathrm{id}$ and the multiplicativity of $(V,[\cdot,\cdot]_{V},\alpha_{V},B_{V})$
that
\[
[x_{0},\mathscr{D}(v)]=[x_{0},\mathscr{D}(v)]_{V}+B_{V}(\mathscr{D}(x_{0}),\mathscr{D}(v))e=0,
\]
\begin{align*}
 & [\alpha_{V}^{j}(x_{0}),\mathscr{D}^{k}(v)]\:(1<j<p-2,\:j\:\mathrm{is}\:\mathrm{even},\:k\in\mathbb{N})\\
= & [x_{0},\mathscr{D}^{k}(v)]\\
= & 0
\end{align*}
and
\begin{align*}
 & [\alpha_{V}^{j}(x_{0}),\mathscr{D}^{k}(v)]\:(1\leq j\leq p-2,\:j\:\mathrm{is}\:\mathrm{odd},\:k\in\mathbb{N})\\
= & [\alpha_{V}(x_{0}),\mathscr{D}^{k}(v)]=[\alpha_{V}(x_{0}),\mathscr{D}^{k}(v)]_{V}+B_{V}(\mathscr{D}(\alpha_{V}(x_{0})),\mathscr{D}^{k}(v))e\\
= & [\alpha_{V}(x_{0}),\alpha_{V}^{2}(\mathscr{D}^{k}(v))]_{V}+B_{V}(\alpha_{V}(\mathscr{D}(x_{0})),\mathscr{D}^{k}(v))e\\
= & \alpha_{V}([x_{0},\alpha_{V}(\mathscr{D}^{k}(v))]_{V})\\
= & 0.
\end{align*}
Therefore, we have $\mathrm{ad}(\alpha^{p-1}(e^{*}))\circ\mathrm{ad}(\alpha^{p-2}(e^{*}))\circ\cdot\cdot\cdot\circ\mathrm{ad}(e^{*})(f)=\mathscr{D}^{p}(v)$
and
\begin{align*}
 & \mathrm{ad}(a_{0}+le+\xi e^{*})\circ\alpha^{p-1}(f)-\mathrm{ad}(\alpha^{p-1}(e^{*}))\circ\mathrm{ad}(\alpha^{p-2}(e^{*}))\circ\cdot\cdot\cdot\circ\mathrm{ad}(e^{*})(f)\\
= & [a_{0}+le+\xi e^{*},je^{*}+j\sum_{1\leq i\leq p-1}\alpha_{V}^{i-1}(x_{0})+\alpha_{V}^{p-1}(v)]-\mathscr{D}^{p}(v)\\
= & j[a_{0},e^{*}]+j\sum_{1\leq i\leq p-1}[a_{0},\alpha_{V}^{i-1}(x_{0})]+j\xi\sum_{1\leq i\leq p-1}[e^{*},\alpha_{V}^{i-1}(x_{0})]\\
 & +[a_{0},\alpha_{V}^{p-1}(v)]+\xi[e^{*},\alpha_{V}^{p-1}(v)]-\mathscr{D}^{p}(v)\\
= & -j\mathscr{D}(a_{0})+j\sum_{1\leq i\leq p-1}([a_{0},\alpha_{V}^{i-1}(x_{0})]_{V}+B_{V}(\mathscr{D}(a_{0}),\alpha_{V}^{i-1}(x_{0}))e)\\
 & +j\xi\sum_{1\leq i\leq p-1}\mathscr{D}(\alpha_{V}^{i-1}(x_{0}))+[a_{0},\alpha_{V}^{p-1}(v)]_{V}+B_{V}(\mathscr{D}(a_{0}),\alpha_{V}^{p-1}(v))e\\
 & +\xi\mathscr{D}(\alpha_{V}^{p-1}(v))-\mathscr{D}^{p}(v)\\
= & j\xi\sum_{1\leq i\leq p-1}\alpha_{V}^{i-1}(\mathscr{D}(x_{0}))+\xi\mathscr{D}\circ\alpha_{V}^{p-1}(v)+\mathrm{ad}_{V}(a_{0})\circ\alpha_{V}^{p-1}(v)-\mathscr{D}^{p}(v)\\
= & (\xi\mathscr{D}\circ\alpha_{V}^{p-1}+\mathrm{ad}_{V}(a_{0})\circ\alpha_{V}^{p-1}-\mathscr{D}^{p})(v)\\
= & 0.
\end{align*}

In addition, we have
\begin{align*}
 & \mathrm{ad}(u^{[p]_{V}}+\mathscr{P}(u)e)\circ\alpha^{p-1}(f)-\mathrm{ad}(\alpha^{p-1}(u))\circ\mathrm{ad}(\alpha^{p-2}(u))\circ\cdot\cdot\cdot\circ\mathrm{ad}(u)(f)\\
= & [u^{[p]_{V}}+\mathscr{P}(u)e,je^{*}+j\sum_{1\leq i\leq p-1}\alpha_{V}^{i-1}(x_{0})+\alpha_{V}^{p-1}(v)]\\
 & -[\alpha^{p-1}(u),[\alpha^{p-2}(u),\cdot\cdot\cdot,[u,je^{*}+v+ke]]]\\
= & j[u^{[p]_{V}},e^{*}]+j\sum_{1\leq i\leq p-1}[u^{[p]_{V}},\alpha_{V}^{i-1}(x_{0})]+[u^{[p]_{V}},\alpha_{V}^{p-1}(v)]\\
 & -[\alpha^{p-1}(u),[\alpha^{p-2}(u),\cdot\cdot\cdot,[\alpha(u),-j\mathscr{D}(u)+[u,v]]]\\
= & -j\mathscr{D}(u^{[p]_{V}})+[u^{[p]_{V}},\alpha_{V}^{p-1}(v)]_{V}+B_{V}(\mathscr{D}(u^{[p]_{V}}),\alpha_{V}^{p-1}(v))e\\
 & +j[\alpha_{V}^{p-1}(u),\cdot\cdot\cdot,[\alpha_{V}(u),\mathscr{D}(u)]_{V}]_{V}+jB_{V}(\mathscr{D}(\alpha_{V}^{p-1}(u)),[\alpha_{V}^{p-2}(u),\cdot\cdot\cdot,[\alpha_{V}(u),\mathscr{D}(u)]_{V})e\\
 & -[\alpha_{V}^{p-1}(u),[\alpha_{V}^{p-2}(u),\cdot\cdot\cdot,[u,v]_{V}]_{V}]_{V}-B_{V}(\mathscr{D}(\alpha_{V}^{p-1}(u)),[\alpha_{V}^{p-2}(u),\cdot\cdot\cdot,[u,v]_{V}]_{V})e\\
= & jB_{V}(\mathscr{D}(u),[\alpha_{V}(u),[u,[\alpha_{V}(u),\cdot\cdot\cdot,[\alpha_{V}(u),\mathscr{D}(u)]_{V})e+B_{V}(\mathscr{D}(u^{[p]_{V}}),\alpha_{V}^{p-1}(v))e\\
 & -B_{V}(\mathscr{D}(\alpha_{V}^{p-1}(u)),[\alpha_{V}^{p-2}(u),\cdot\cdot\cdot,[u,v]_{V}]_{V})e\\
= & B_{V}([\alpha_{V}^{p-1}(u),\cdot\cdot\cdot,[\alpha_{V}(u),\mathscr{D}(u)]_{V}]_{V},\alpha_{V}^{p-1}(v))e\\
 & -B_{V}(\mathscr{D}(\alpha_{V}^{p-1}(u)),[\alpha_{V}^{p-2}(u),\cdot\cdot\cdot,[u,v]_{V}]_{V})e\\
= & B_{V}([u,[\alpha_{V}(u),[u,\cdot\cdot\cdot,[\alpha_{V}(u),\mathscr{D}(u)]_{V}]_{V}]_{V},v)e\\
 & -B_{V}(\mathscr{D}(u),[\alpha_{V}(u),[u,[\alpha_{V}(u),\cdot\cdot\cdot,[u,v]_{V}]_{V}]_{V}]_{V})e\\
= & B_{V}(\mathscr{D}(u),[\alpha_{V}(u),[u,[\alpha_{V}(u),\cdot\cdot\cdot,[u,v]_{V}]_{V}]_{V}]_{V})e\\
 & -B_{V}(\mathscr{D}(u),[\alpha_{V}(u),[u,[\alpha_{V}(u),\cdot\cdot\cdot,[u,v]_{V}]_{V}]_{V}]_{V})e\\
= & 0.
\end{align*}

Moreover, for any $u,v\in V$, we have
\begin{align*}
 & (u+v)^{[p]_{L}}\\
= & (u+v)^{[p]_{V}}+\mathscr{P}(u+v)e\\
= & u^{[p]_{V}}+v^{[p]_{V}}+\sum_{1\leq i\leq p-1}s_{i}^{V}(u,v)+\mathscr{P}(u+v)e\\
= & (u^{[p]_{L}}-\mathscr{P}(u)e)+(v^{[p]_{L}}-\mathscr{P}(v)e)+\sum_{1\leq i\leq p-1}s_{i}^{L}(u,v)-\sum_{1\leq i\leq p-1}\eta_{i}^{V}(u,v)e+\mathscr{P}(u+v)e\\
= & (u^{[p]_{L}}+v^{[p]_{L}}+\sum_{1\leq i\leq p-1}s_{i}^{L}(u,v))+(\mathscr{P}(u+v)-\mathscr{P}(u)-\mathscr{P}(v)-\sum_{1\leq i\leq p-1}\eta_{i}^{V}(u,v))e\\
= & u^{[p]_{L}}+v^{[p]_{L}}+\sum_{1\leq i\leq p-1}s_{i}^{L}(u,v).
\end{align*}

To sum up, $[p]_{L}$ is a $p$-structure on $L$. The proof is complete.
\end{proof}
The following theorem is the converse of Theorem 4.4.

\begin{theorem} Let $(L,[\cdot,\cdot],\alpha,B)$ be an irreducible
restricted quadratic Hom-Lie algebra such that $\mathrm{dim}L>1$
and $\mathfrak{z}(L)\neq\{0\}$. If $0\neq e\in\mathfrak{z}(L)$ and
$\mathscr{E}=\mathrm{Span}\{e\}$ such that $\mathscr{E}{}^{\perp}$
is a $p$-ideal, then $(L,[\cdot,\cdot],\alpha,B)$ is the double
extension of a restricted involutive multiplicative quadratic Hom-Lie
algebra $(V,[\cdot,\cdot]_{V},\alpha_{V},B_{V})$ by means of $\mathscr{D}$,
where $\mathscr{D}\in\mathrm{Der}^{p}(V)$ and $\mathscr{D}$ has
$p$-property.

\end{theorem}
\begin{proof}
According to Lemma 2.8, we have $\alpha(\mathfrak{z}(L))\subset\mathfrak{z}(L)$.
Since $0\neq e\in\mathfrak{z}(L)$, the vecter space $\mathscr{E}$
is an ideal of $L$. Then there exists $\lambda\in\mathbb{F}$ such
that $\alpha(e)=\lambda e$. Moreover, since $\mathscr{E}$ is an
ideal of $L$, $L$ is irreducible and $\mathrm{dim}L>1$, we have
$B(e,e)=0$. It follows from $B$ being nondegenerate that there exists
$0\neq e^{*}\in L$ such that $B(e^{*},e)=1$ and $B(e^{*},e^{*})=0$.

Let $\mathscr{E}^{*}=\mathrm{Span}\{e^{*}\}$ and $\Gamma=\mathscr{E}\oplus\mathscr{E}^{*}$,
then $\Gamma$ is nondegenerate (i.e. $B|_{\Gamma\times\Gamma}$ is
nondegenerate) and $L=\Gamma\oplus\Gamma^{\bot}$. Furthermore, set
$V:=\Gamma^{\bot}$, then $\mathscr{E}{}^{\perp}=\mathscr{E}\oplus V$,
$B|_{V\times V}$ is nondegenerate and there exists a decomposition
$L=\mathscr{E}{}^{*}\oplus V\oplus\mathscr{E}$.

It follows from $\mathscr{E}{}^{\perp}$ being an ideal that $\alpha(\mathscr{E}{}^{\perp})\subset\mathscr{E}{}^{\perp}$.
Then there exists a linear map $\alpha_{V}:V\rightarrow V$ and a
linear form $\kappa:V\rightarrow\mathbb{F}$ such that $\alpha(u)=\alpha_{V}(u)+\kappa(u)e,$
for every $u\in V$. In addition, there exists $l\in\mathbb{F},x_{0}\in V$
and $\lambda_{0}\in\mathbb{F}$ such that $\alpha(e^{*})=le^{*}+x_{0}+\lambda_{0}e$.

Be analogue to Theorem 6.6 in \citep{BM}, it can be proved that there
exists a multiplicative quadratic Hom-Lie algebra structure on $V$
such that $L$ is the double extension of $V$ as given in Theorem
4.1.

Denote by $B_{V}:=B|_{V\times V}$ the nondegenerate invariant bilinear
form on $V$. We will prove that there exists a $p$-structure on
$V$.

Since $V\subset\mathscr{E}{}^{\perp}$, we have
\[
v^{[p]_{L}}\in\mathscr{E}_{p}{}^{\perp}=\mathscr{E}{}^{\perp}=\mathscr{E}\oplus V,\:\mathrm{for\:every\:}v\in V.
\]
It follows that $v^{[p]_{L}}=\mathscr{P}(v)e+s(v)$, where $s(v)\in V$.

The map $s:V\rightarrow V,v\mapsto s(v)$ will be proved to be a $p$-structure
on V.

Since $(\gamma v)^{[p]_{L}}=\gamma^{p}v{}^{[p]_{L}}$, we have $\mathscr{P}(\gamma v)=\gamma^{p}\mathscr{P}(v)$
and $s(\gamma v)=\gamma^{p}s(v)$.

Besides, for any $u,v\in V$, we have
\begin{align*}
0= & [u^{[p]_{L}},\alpha^{p-1}(v)]-[\alpha^{p-1}(u),[\alpha^{p-2}(u),\cdot\cdot\cdot,[u,v]]]\\
= & [\mathscr{P}(u)e+s(u),\alpha_{V}^{p-1}(v)+\sum_{0\leq i\leq p-2}B_{V}(x_{0},\alpha_{V}^{(p-2)-i}(v))e]\\
 & -[\alpha_{V}^{p-1}(u)+\sum_{0\leq i\leq p-2}B_{V}(x_{0},\alpha_{V}^{(p-2)-i}(u))e,[\alpha_{V}^{p-2}(u)\\
 & +\sum_{0\leq i\leq p-3}B_{V}(x_{0},\alpha_{V}^{(p-3)-i}(u))e,\cdot\cdot\cdot,[\alpha_{V}(u)+B_{V}(x_{0},u)e,[u,v]_{V}+B_{V}(\mathscr{D}(u),v)e]]]\\
= & [s(u),\alpha_{V}^{p-1}(v)]-[\alpha_{V}^{p-1}(u),[\alpha_{V}^{p-2}(u),\cdot\cdot\cdot,[u,v]_{V}]]\\
= & [s(u),\alpha_{V}^{p-1}(v)]_{V}+B_{V}(\mathscr{D}(s(u)),\alpha_{V}^{p-1}(v))e-[\alpha_{V}^{p-1}(u),[\alpha_{V}^{p-2}(u),\cdot\cdot\cdot,[u,v]_{V}]_{V}]_{V}\\
 & -B_{V}(\mathscr{D}(\alpha_{V}^{p-1}(u)),[\alpha_{V}^{p-2}(u),\cdot\cdot\cdot,[\alpha_{V}(u),[u,v]_{V}]_{V}]_{V})e\\
= & [s(u),\alpha_{V}^{p-1}(v)]_{V}-[\alpha_{V}^{p-1}(u),[\alpha_{V}^{p-2}(u),\cdot\cdot\cdot,[u,v]_{V}]_{V}]_{V}+B_{V}(\mathscr{D}(s(u)),v)e\\
 & -B_{V}([u,\cdot\cdot\cdot,[\alpha_{V}^{p-3}(u),[\alpha_{V}^{p-2}(u),\mathscr{D}(u),]_{V}]_{V}]_{V},v)e\\
= & [s(u),\alpha_{V}^{p-1}(v)]_{V}-[\alpha_{V}^{p-1}(u),[\alpha_{V}^{p-2}(u),\cdot\cdot\cdot,[u,v]_{V}]_{V}]_{V}+B_{V}(\mathscr{D}(s(u)),v)e\\
 & -B_{V}([\alpha_{V}^{p-1}(u),[\alpha_{V}^{p-2}(u),\cdot\cdot\cdot,[\alpha_{V}(u),\mathscr{D}(u),]_{V}]_{V}]_{V},v)e
\end{align*}

Since $B_{V}$ is nondegenerate, we obtain
\[
[s(u),\alpha_{V}^{p-1}(v)]_{V}=[\alpha_{V}^{p-1}(u),[\alpha_{V}^{p-2}(u),\cdot\cdot\cdot,[u,v]_{V}]_{V}]_{V}
\]
and
\[
\mathscr{D}(s(u))=[\alpha_{V}^{p-1}(u),[\alpha_{V}^{p-2}(u),\cdot\cdot\cdot,[\alpha_{V}(u),\mathscr{D}(u)]_{V}]_{V}]_{V}.
\]

Based on Lemma 4.3, we have
\begin{align*}
0= & (u+v)^{[p]_{L}}-u^{[p]_{L}}-v^{[p]_{L}}-\sum_{1\leq i\leq p-1}s_{i}^{L}(u,v)\\
= & (s(u+v)+\mathscr{P}(u+v))-(s(u)+\mathscr{P}(u))-(s(v)+\mathscr{P}(v))\\
 & -\sum_{1\leq i\leq p-1}(s_{i}^{V}(u,v)+\eta_{i}^{V}(u,v)e)\\
= & (s(u+v)-s(u)-s(v)-\sum_{1\leq i\leq p-1}s_{i}^{V}(u,v))\\
 & +(\mathscr{P}(u+v)-\mathscr{P}(u)-\mathscr{P}(v)-\sum_{1\leq i\leq p-1}\eta_{i}^{V}(u,v))e.
\end{align*}

It follows that $s(u+v)=s(u)+s(v)+\sum_{1\leq i\leq p-1}s_{i}^{V}(u,v)$
and $\mathscr{P}(u+v)=\mathscr{P}(u)+\mathscr{P}(v)+\sum_{1\leq i\leq p-1}\eta_{i}^{V}(u,v)$.

It implies that $s$ defines a $p$-structure on $V$, $\mathscr{D}\in\mathrm{Der}^{res}(V)$
and the map $\mathscr{P}$ satisfies Eq. (4.6).

Assume that $e^{[p]_{L}}=u_{0}+me+\delta e^{*}$, where $m,\delta\in\mathbb{F}$.
For every $v\in V$, we have
\begin{align*}
0 & =[e^{[p]_{L}},\alpha^{p-1}(v)]-[\alpha^{p-1}(e),[\alpha^{p-2}(e),\cdot\cdot\cdot,[e,v]]]\\
 & =[u_{0}+me+\delta e^{*},\alpha_{V}^{p-1}(v)+\sum_{0\leq i\leq p-2}B_{V}(x_{0},\alpha_{V}^{(p-2)-i}(v))e]\\
 & =[u_{0},\alpha_{V}^{p-1}(v)]+\delta[e^{*},\alpha_{V}^{p-1}(v)]\\
 & =[u_{0},\alpha_{V}^{p-1}(v)]_{V}+B_{V}(\mathscr{D}(u_{0}),\alpha_{V}^{p-1}(v))e+\delta\mathscr{D}(\alpha_{V}^{p-1}(v))\\
 & =([u_{0},v]_{V}+\delta\mathscr{D}(v))+B_{V}(\mathscr{D}(u_{0}),v)e.
\end{align*}

Therefore, $\delta=0$, $u_{0}\in\mathfrak{z}(V)$ and $\mathscr{D}(u_{0})=0$.

Assume now that $(e^{*})^{[p]_{L}}=a_{0}+le+\xi e^{*}$, where $a_{0}\in V,l,\xi\in\mathbb{F}$.
For any $v\in V$, we have
\begin{align*}
0= & [(e^{*})^{[p]_{L}},\alpha^{p-1}(v)]-[\alpha^{p-1}(e^{*}),[\alpha^{p-2}(e^{*}),\cdot\cdot\cdot,[e^{*},v]]]\\
= & [a_{0}+le+\xi e^{*},\alpha_{V}^{p-1}(v)+\sum_{0\leq i\leq p-2}B_{V}(x_{0},\alpha_{V}^{(p-2)-i}(v))e]-\mathscr{D}^{p}(v)\\
= & [a_{0},\alpha_{V}^{p-1}(v)]+\xi[e^{*},\alpha_{V}^{p-1}(v)]-\mathscr{D}^{p}(v)\\
= & [a_{0},\alpha_{V}^{p-1}(v)]_{V}+B_{V}(\mathscr{D}(a_{0}),\alpha_{V}^{p-1}(v))e+\xi\mathscr{D}(\alpha_{V}^{p-1}(v))-\mathscr{D}^{p}(v)\\
= & (\xi\mathscr{D}\circ\alpha_{V}^{p-1}+\mathrm{ad}_{V}(a_{0})\circ\alpha_{V}^{p-1}-\mathscr{D}^{p})(v)+B_{V}(\mathscr{D}(a_{0}),\alpha_{V}^{p-1}(v))e.
\end{align*}

Therefore, $\mathscr{D}^{p}=\xi\mathscr{D}\circ\alpha_{V}^{p-1}+\mathrm{ad}_{V}(a_{0})\circ\alpha_{V}^{p-1}$
and $\mathscr{D}(a_{0})=0$.

The proof is complete.
\end{proof}
Be analogue to the case in characteristic 2, here we give the involutive
double extension theorem of an involutive quadratic Hom-Lie algebra
$V$ as follows:

\begin{theorem} (Involutive Double Extension Theorem) Let $(V,[\cdot,\cdot]_{V},\alpha_{V},B_{V})$
be an involutive quadratic Hom-Lie algebra and $(A,[\cdot,\cdot]_{A},\alpha_{A})$
an involutive Hom-Lie algebra.

Let $\phi:A\rightarrow\mathrm{End}_{a}(V,B_{V})$, $a\mapsto\phi(a)$
be a representation of $A$ on $(V,[\cdot,\cdot]_{V},\alpha_{V})$
such that
\begin{align*}
\alpha_{V}\circ\phi(a)([x,y]_{V})= & [\phi(a)\circ\alpha_{V}(x),y]_{V}+[x,\phi(a)\circ\alpha_{V}(y)]_{V},\:\forall x,y\in V,\\
\phi\circ\alpha_{A}(a)= & \alpha_{V}\circ\phi(a)\circ\alpha_{V},\:\forall a\in A
\end{align*}
and $\psi:V\times V\rightarrow A^{*},(x,y)\mapsto\psi(x,y)$ defined
by $\psi(x,y)(a)=B_{V}(\phi(a)x,y),\forall x,y\in V,\forall a\in A$.

(1) There exists an involutive multiplicative Hom-Lie algebra structure
on $L=A^{*}\oplus V\oplus A$, where the bracket $[\cdot,\cdot]:L\times L\rightarrow L$
is defined as follows (for any $f,f'\in A^{*},\:x,x'\in V$ and $a,a'\in A$):

\begin{align*}
[f+x+a,f'+x'+a']= & -f\circ\mathrm{ad}_{A}(a')+f'\circ\mathrm{ad}_{A}(a)+\psi(x,x')+[x,x']_{V}\\
 & +\phi(a)(x')-\phi(a')(x)+[a,a']_{A}
\end{align*}
and the linear map $\alpha:L\rightarrow L$ is defined by

\[
\alpha(f+x+a)=f\circ\alpha_{A}+\alpha_{V}(x)+\alpha_{A}(a).
\]

(2) If the bilinear form $B_{\sigma}:L\times L\rightarrow\mathbb{F}$
is defined by

\[
B_{\sigma}(f+x+a,f'+x'+a')=B_{V}(x,x')+f(a')+f'(a)+\sigma(a,a'),
\]
where $\sigma$ is a symmetric nondegenerate invariant bilinear form
on $(A,[\cdot,\cdot]_{A},\alpha_{A})$ such that
\[
\sigma(\alpha_{A}(a),a')=\sigma(a,\alpha_{A}(a')),\forall a,a'\in A,
\]
then the quadruple $(L,[\cdot,\cdot],\alpha,B_{\sigma})$ is an involutive
multiplicative quadratic Hom-Lie algebra.

\end{theorem}
\begin{proof}
The proof is similar to the proof of Theorem 6.10 in \citep{BM}.
\end{proof}
Next, the adapted isomorphism of the double extensions of $(V,[\cdot,\cdot]_{V},\alpha_{V},B_{V})$
will be studied. Here, we will use the same notations as the corresponding
part in Sec. 3.

Be similar to \citep{FS} and the case for $p=2$ in Sec. 3, we obtain
the following conclusion:

Let $B_{V}$ be $\mathscr{D}$- and $\widetilde{\mathscr{D}}$-invariant,
where $\mathscr{D},\mathscr{\widetilde{D}}\in\mathrm{Der}(V)$. Then
$\pi:L\rightarrow\widetilde{L}$ is an adapted isomorphism if and
only if there exists exactly an automorphism $\pi_{0}:V\rightarrow V,\gamma\in\mathbb{F}$\textbackslash\{0\}
and $t_{\pi}\in V$ (depending only on $\pi$) such that

\begin{align*}
\pi_{0}^{-1}\mathscr{\widetilde{D}}\pi_{0} & =\gamma\mathscr{D}+\mathrm{ad}_{V}(t_{\pi}),\\
\widetilde{\lambda} & =\lambda,\\
\pi_{0}(\alpha_{V}(t_{\pi})-\lambda t_{\pi}) & =\widetilde{x}_{0}-\gamma\pi_{0}(x_{0}),\\
B_{V}(\widetilde{x}_{0},\pi_{0}(t_{\pi}))+\gamma B_{V}(x_{0},t_{\pi}) & =\widetilde{\lambda}_{0}-\gamma^{2}\lambda_{0},\\
B_{V}(\widetilde{x}_{0},\pi_{0}(v))-\gamma B_{V}(x_{0},v)) & =B_{V}(t_{\pi},\alpha_{V}(v)-\lambda v)
\end{align*}
and

\begin{align}
\pi(u) & =\pi_{0}(u)+B_{V}(t_{\pi},u)\widetilde{e},\forall u\in V,\nonumber \\
\pi(e) & =\gamma\widetilde{e},\\
\pi(e^{*}) & =\gamma^{-1}(\widetilde{e}^{*}-\pi_{0}(t_{\pi})-\frac{1}{2}B_{V}(t_{\pi},t_{\pi})\widetilde{e}).\nonumber
\end{align}

Suppose now that $(V,[\cdot,\cdot]_{V},\alpha_{V},B_{V})$ is a restricted
involutive quadratic Hom-Lie algebra with a $p$-structure $[p]_{V}$.
According to Theorem 4.4, the $p$-structure $[p]_{V}$ can be extended
to any double extension of $V$. Denote by $[p]_{L}$ (resp. $[p]_{\widetilde{L}}$)
the $p$-structure on $L$ (resp. $\widetilde{L}$) written in terms
of $\xi,m,l,a_{0},u_{0}$ and $\mathscr{P}$ (resp. $\widetilde{\xi},\widetilde{m},\widetilde{l},\widetilde{a}_{0},\widetilde{u}_{0}$
and $\mathscr{\widetilde{P}}$).

We will study the equivalence class of $p$-structures on double extensions.
To this end, we should first make some preparations.

For $x,y\in\widetilde{L},k\in\mathbb{F}$, $l\in\mathbb{N}$ and $l\geqslant3$,
let
\[
\mathrm{ad}(\alpha^{l-2}(kx+y))\circ\mathrm{ad}(\alpha^{l-3}(kx+y))\circ\cdots\circ\mathrm{ad}(kx+y)(x)=\sum_{i=1}^{l-1}\Phi_{i}^{l}(x,y)k^{i-1}.
\]

\begin{lemma} Let $x,y\in\widetilde{L}$ and $l\geqslant4$, then
we have

\[
\Phi_{1}^{3}(x,y)=[\alpha(y),[y,x]],\Phi_{2}^{3}(x,y)=[\alpha(x),[y,x]]
\]

and
\begin{equation}
\Phi_{i}^{l}(x,y)=\begin{cases}
\mathrm{ad}(\alpha^{l-2}(y))\circ\mathrm{ad}(\alpha^{l-3}(y))\circ\cdots\circ\mathrm{ad}(\alpha(y))([y,x]), & i=1,\\
\mathrm{ad}(\alpha^{l-2}(y))(\Phi_{i}^{l-1}(x,y))+\mathrm{ad}(\alpha^{l-2}(x))(\Phi_{i-1}^{l-1}(x,y)), & 2\leqslant i\leqslant l-2,\\
\mathrm{ad}(\alpha^{l-2}(x))\circ\mathrm{ad}(\alpha^{l-3}(x))\circ\cdots\circ\mathrm{ad}(\alpha(x))([y,x]), & i=l-1.
\end{cases}
\end{equation}
\end{lemma}
\begin{proof}
It follows from

\[
\mathrm{ad}(\alpha(kx+y))\circ\mathrm{ad}(kx+y)(x)=[\alpha(y),[y,x]]+[\alpha(x),[y,x]]k,
\]
and
\[
\mathrm{ad}(\alpha(kx+y))\circ\mathrm{ad}(kx+y)(x)=\Phi_{1}^{3}(x,y)+\Phi_{2}^{3}(x,y)k,
\]
that $\Phi_{1}^{3}(x,y)=[\alpha(y),[y,x]]$ and $\Phi_{2}^{3}(x,y)=[\alpha(x),[y,x]]$.

Moreover, we have

\begin{align*}
 & \mathrm{ad}(\alpha^{l-2}(kx+y))\circ\mathrm{ad}(\alpha^{l-3}(kx+y))\circ\cdots\circ\mathrm{ad}(kx+y)(x)\\
= & \mathrm{ad}(\alpha^{l-2}(kx+y))(\sum_{i=1}^{l-2}\Phi_{i}^{l-1}(x,y)k^{i-1})\\
= & [k\alpha^{l-2}(x),\sum_{i=1}^{l-2}\Phi_{i}^{l-1}(x,y)k^{i-1}]+[\alpha^{l-2}(y),\sum_{i=1}^{l-2}\Phi_{i}^{l-1}(x,y)k^{i-1}]\\
= & \sum_{i=1}^{l-2}\mathrm{ad}(\alpha^{l-2}(x))(\Phi_{i}^{l-1}(x,y))k^{i}+\sum_{i=1}^{l-2}\mathrm{ad}(\alpha^{l-2}(y))(\Phi_{i}^{l-1}(x,y))k^{i-1}\\
= & \sum_{i=2}^{l-1}\mathrm{ad}(\alpha^{l-2}(x))(\Phi_{i-1}^{l-1}(x,y))k^{i-1}+\sum_{i=1}^{l-2}\mathrm{ad}(\alpha^{l-2}(y))(\Phi_{i}^{l-1}(x,y))k^{i-1}\\
= & \mathrm{ad}(\alpha^{l-2}(y))(\Phi_{1}^{l-1}(x,y))+\mathrm{ad}(\alpha^{l-2}(x))(\Phi_{l-2}^{l-1}(x,y))k^{l-2}\\
 & +\sum_{i=2}^{l-2}(\mathrm{ad}(\alpha^{l-2}(x))(\Phi_{i-1}^{l-1}(x,y))+\mathrm{ad}(\alpha^{l-2}(y))(\Phi_{i}^{l-1}(x,y)))k^{i-1}.
\end{align*}

Then Eq. (4.8) is obtained by comparing the coefficients and induction.

The proof is complete.
\end{proof}
Set $x:=\widetilde{e}^{\ast},y:=-\pi_{0}(t_{\pi})$. The following
lemma can be obtained by a direct computation.

\begin{lemma} For any $f\in\widetilde{L}$ and $k\in\mathbb{N}$,
we have

(1) $\mathrm{ad}(\alpha^{k}(x))(f)=\mathrm{ad}(x+\sum_{i=1}^{k-1}\alpha_{V}^{i}(\widetilde{x}_{0}))(f)$.

(2) $\mathrm{ad}(\alpha^{k}(y))(f)=\mathrm{ad}(\alpha_{V}^{k}(y))(f)$
.

\end{lemma}

\begin{proposition} Let $l\geq4$, then we have the following iterative
formula
\begin{equation}
\Phi_{1}^{l}(x,y)=\Phi_{1}^{l}(x,y)_{\mathrm{I}}+\Phi_{1}^{l}(x,y)_{\mathrm{II}}\widetilde{e},
\end{equation}
where $\Phi_{1}^{l}(x,y)_{\mathrm{I}}=[\alpha_{V}^{l-2}(y),\Phi_{1}^{l-1}(x,y)_{\mathrm{I}}]_{V}$,
$\Phi_{1}^{l}(x,y)_{\mathrm{II}}=B_{V}(\widetilde{\mathscr{D}}(\alpha_{V}^{l-2}(y)),\Phi_{1}^{l-1}(x,y)_{\mathrm{I}})$
and $\Phi_{1}^{3}(x,y)_{\mathrm{I}}=[\alpha_{V}(y),-\widetilde{\mathscr{D}}(y)]_{V}$.

\end{proposition}
\begin{proof}
According to Lemmas 4.7 and 4.8, we have
\begin{align*}
\Phi_{1}^{3}(x,y) & =[\alpha(y),[y,x]]=[\alpha_{V}(y),-\widetilde{\mathscr{D}}(y)]\\
 & =[\alpha_{V}(y),-\widetilde{\mathscr{D}}(y)]_{V}+B_{V}(\widetilde{\mathscr{D}}(\alpha_{V}(y)),-\widetilde{\mathscr{D}}(y))\widetilde{e}.
\end{align*}

Let $\Phi_{1}^{3}(x,y)=\Phi_{1}^{3}(x,y)_{\mathrm{I}}+\Phi_{1}^{3}(x,y)_{\mathrm{II}}\widetilde{e}$,
then $\Phi_{1}^{3}(x,y)_{\mathrm{I}}=[\alpha_{V}(y),-\widetilde{\mathscr{D}}(y)]_{V}$.

For $i=4$, Eq. (4.9) holds by a direct computation.

Suppose now that Eq. (4.9) holds for $i=l-1$. Then for $i=l$, we
have
\begin{align*}
\Phi_{1}^{l}(x,y)= & \mathrm{ad}(\alpha^{l-2}(y))\circ\mathrm{ad}(\alpha^{l-3}(y))\circ\cdots\circ\mathrm{ad}(\alpha(y))([y,x])\\
= & \mathrm{ad}(\alpha^{l-2}(y))\Phi_{1}^{l-1}(x,y)\\
= & \mathrm{ad}(\alpha^{l-2}(y))(\Phi_{1}^{l-1}(x,y)_{\mathrm{I}}+\Phi_{1}^{l-1}(x,y)_{\mathrm{II}}\widetilde{e})\\
= & \mathrm{ad}(\alpha_{V}^{l-2}(y))(\Phi_{1}^{l-1}(x,y)_{\mathrm{I}}+\Phi_{1}^{l-1}(x,y)_{\mathrm{II}}\widetilde{e})\\
= & [\alpha_{V}^{l-2}(y),\Phi_{1}^{l-1}(x,y)_{\mathrm{I}}]\\
= & [\alpha_{V}^{l-2}(y),\Phi_{1}^{l-1}(x,y)_{\mathrm{I}}]_{V}+B_{V}(\widetilde{\mathscr{D}}(\alpha_{V}^{l-2}(y)),\Phi_{l}^{l-1}(x,y)_{\mathrm{I}})\widetilde{e}\\
= & \Phi_{1}^{l}(x,y)_{\mathrm{I}}+\Phi_{1}^{l}(x,y)_{\mathrm{II}}\widetilde{e},
\end{align*}
which implies that Eq. (4.9) holds for $i=l$.

According to induction, Eq. (4.9) holds for all $l\geq4$.

The proof is complete.
\end{proof}
Similarly, the following two theorems can be proved.

\begin{proposition} Let $l\geq4$ and $2\leqslant i\leqslant l-2$
, then we have the following iterative formula
\[
\Phi_{i}^{l}(x,y)=\Phi_{i}^{l}(x,y)_{\mathrm{I}}+\Phi_{i}^{l}(x,y)_{\mathrm{I}\mathrm{I}}\widetilde{e},
\]
where
\[
\Phi_{i}^{l}(x,y)_{\mathrm{I}}=[\alpha_{V}^{l-2}(y),\Phi_{i}^{l-1}(x,y)_{\mathrm{I}}]_{V}+\mathscr{\widetilde{D}}(\Phi_{i-1}^{l-1}(x,y)_{\mathrm{I}})+[\sum_{i=1}^{l-3}\alpha_{V}^{i}(\widetilde{x}_{0}),\Phi_{i-1}^{l-1}(x,y)_{\mathrm{I}}]_{V},
\]
\[
\Phi_{i}^{l}(x,y)_{\mathrm{I}\mathrm{I}}=B_{V}(\mathscr{\widetilde{D}}(\alpha_{V}^{l-2}(y)),\Phi_{i}^{l-1}(x,y)_{\mathrm{I}})+B_{V}(\mathscr{\widetilde{D}}(\sum_{i=1}^{l-3}\alpha_{V}^{i}(\widetilde{x}_{0})),\Phi_{i-1}^{l-1}(x,y)_{\mathrm{I}})
\]
\[
\Phi_{1}^{3}(x,y)_{\mathrm{I}}=[\alpha_{V}(y),-\mathscr{\widetilde{D}}(y)]_{V},\:\Phi_{2}^{3}(x,y)_{\mathrm{I}}=-\mathscr{\widetilde{D}}^{2}(y).
\]

\end{proposition}

\begin{proposition} Let $l\geq4$, then we have the following iterative
formula
\[
\Phi_{l-1}^{l}(x,y)=\Phi_{l-1}^{l}(x,y)_{\mathrm{I}}+\Phi_{l-1}^{l}(x,y)_{\mathrm{II}}\widetilde{e},
\]
where
\[
\Phi_{l-1}^{l}(x,y)_{\mathrm{I}}=\mathscr{\widetilde{D}}(\Phi_{l-2}^{l-1}(x,y)_{\mathrm{I}})+[\sum_{i=1}^{l-3}\alpha_{V}^{i}(\widetilde{x}_{0}),\Phi_{l-2}^{l-1}(x,y)_{\mathrm{I}}]_{V},
\]
\[
\Phi_{l-1}^{l}(x,y)_{\mathrm{II}}=B_{V}(\mathscr{\widetilde{D}}(\sum_{i=1}^{l-3}\alpha_{V}^{i}(\widetilde{x}_{0})),\Phi_{l-2}^{l-1}(x,y)_{\mathrm{I}})
\]
and
\[
\Phi_{2}^{3}(x,y)_{\mathrm{I}}=-\mathscr{\widetilde{D}}^{2}(y).
\]

\end{proposition}

Based on the above discussion, we obtain

\begin{theorem} For $k\in\mathbb{F}$, denote by $s_{i}^{\widetilde{L}}(\widetilde{e}^{*},-\pi_{0}(t_{\pi}))$
the coefficient of the following expansion

\begin{align*}
 & \mathrm{ad}(\alpha^{p-2}(k\widetilde{e}^{*}-\pi_{0}(t_{\pi})))\circ\mathrm{ad}(\alpha^{p-3}(k\widetilde{e}^{*}-\pi_{0}(t_{\pi})))\circ\cdot\cdot\cdot\circ\mathrm{ad}(k\widetilde{e}^{*}-\pi_{0}(t_{\pi}))(\widetilde{e}^{*})\\
= & \sum_{i=1}^{p-1}is_{i}^{\widetilde{L}}(\widetilde{e}^{*},-\pi_{0}(t_{\pi}))k^{i-1}.
\end{align*}
Then

(1) for $p=3$, we have
\[
\sum_{i=1}^{2}s_{i}^{\widetilde{L}}(\widetilde{e}^{*},-\pi_{0}(t_{\pi}))=\sum_{i=1}^{2}\frac{1}{i}\Phi_{i}^{3}(\widetilde{e}^{*},-\pi_{0}(t_{\pi}))_{\mathrm{I}}+\sum_{i=1}^{2}\frac{1}{i}\Phi_{i}^{3}(\widetilde{e}^{*},-\pi_{0}(t_{\pi}))_{\mathrm{II}}\widetilde{e},
\]
where
\begin{align*}
\Phi_{1}^{3}(\widetilde{e}^{*},-\pi_{0}(t_{\pi}))_{\mathrm{I}}= & -[\alpha_{V}(\pi_{0}(t_{\pi})),\widetilde{\mathscr{D}}(\pi_{0}(t_{\pi}))]_{V},\Phi_{2}^{3}(\widetilde{e}^{*},-\pi_{0}(t_{\pi}))_{\mathrm{I}}=\widetilde{\mathscr{D}}^{2}(\pi_{0}(t_{\pi})),
\end{align*}
\[
\Phi_{1}^{3}(\widetilde{e}^{*},-\pi_{0}(t_{\pi}))_{\mathrm{II}}=-B_{V}(\widetilde{\mathscr{D}}(\alpha_{V}(\pi_{0}(t_{\pi}))),\widetilde{\mathscr{D}}(\pi_{0}(t_{\pi}))),\Phi_{2}^{3}(\widetilde{e}^{*},-\pi_{0}(t_{\pi}))_{\mathrm{II}}=0.
\]

(2) for $p>3$, we have

\[
\sum_{i=1}^{p-1}s_{i}^{\widetilde{L}}(\widetilde{e}^{*},-\pi_{0}(t_{\pi}))=\sum_{i=1}^{p-1}\frac{1}{i}\Phi_{i}^{p}(\widetilde{e}^{*},-\pi_{0}(t_{\pi}))_{\mathrm{I}}+\sum_{i=1}^{p-1}\frac{1}{i}\Phi_{i}^{p}(\widetilde{e}^{*},-\pi_{0}(t_{\pi}))_{\mathrm{II}}\widetilde{e},
\]
where
\begin{align*}
\Phi_{1}^{p}(\widetilde{e}^{*},-\pi_{0}(t_{\pi}))_{\mathrm{I}}= & [\alpha_{V}^{p-2}(-\pi_{0}(t_{\pi})),\Phi_{1}^{p-1}(\widetilde{e}^{*},-\pi_{0}(t_{\pi}))_{\mathrm{I}}]_{V},\\
\Phi_{i}^{p}(\widetilde{e}^{*},-\pi_{0}(t_{\pi}))_{\mathrm{I}}= & [\alpha_{V}^{p-2}(-\pi_{0}(t_{\pi})),\Phi_{i}^{p-1}(\widetilde{e}^{*},-\pi_{0}(t_{\pi}))_{\mathrm{I}}]_{V}+\widetilde{\mathscr{D}}(\Phi_{i-1}^{p-1}(\widetilde{e}^{*},-\pi_{0}(t_{\pi}))_{\mathrm{I}})\\
 & +[\sum_{i=1}^{p-3}\alpha_{V}^{i}(\widetilde{x}_{0}),\Phi_{i-1}^{p-1}(\widetilde{e}^{*},-\pi_{0}(t_{\pi}))_{\mathrm{I}}]_{V},(2\leq i\leq l-2),\\
\Phi_{p-1}^{p}(\widetilde{e}^{*},-\pi_{0}(t_{\pi}))_{\mathrm{I}}= & \widetilde{\mathscr{D}}(\Phi_{p-2}^{p-1}(\widetilde{e}^{*},-\pi_{0}(t_{\pi}))_{\mathrm{I}})+[\sum_{i=1}^{p-3}\alpha_{V}^{i}(\widetilde{x}_{0}),\Phi_{p-2}^{p-1}(\widetilde{e}^{*},-\pi_{0}(t_{\pi}))_{\mathrm{I}}]_{V},
\end{align*}
\begin{align*}
\Phi_{1}^{p}(\widetilde{e}^{*},-\pi_{0}(t_{\pi}))_{\mathrm{II}}= & B_{V}(\widetilde{\mathscr{D}}(\alpha_{V}^{p-2}(-\pi_{0}(t_{\pi}))),\Phi_{1}^{p-1}(\widetilde{e}^{*},-\pi_{0}(t_{\pi}))_{\mathrm{I}}),\\
\Phi_{i}^{p}(\widetilde{e}^{*},-\pi_{0}(t_{\pi}))_{\mathrm{II}}= & B_{V}(\widetilde{\mathscr{D}}(\alpha_{V}^{p-2}(-\pi_{0}(t_{\pi}))),\Phi_{i}^{p-1}(\widetilde{e}^{*},-\pi_{0}(t_{\pi}))_{\mathrm{I}})\\
 & +B_{V}(\widetilde{\mathscr{D}}(\sum_{i=1}^{p-3}\alpha_{V}^{i}(\widetilde{x}_{0})),\Phi_{i-1}^{p-1}(\widetilde{e}^{*},-\pi_{0}(t_{\pi}))_{\mathrm{I}}),(2\leq i\leq l-2),\\
\Phi_{p-1}^{p}(\widetilde{e}^{*},-\pi_{0}(t_{\pi}))_{\mathrm{II}}= & B_{V}(\widetilde{\mathscr{D}}(\sum_{i=1}^{p-3}\alpha_{V}^{i}(\widetilde{x}_{0})),\Phi_{p-2}^{p-1}(\widetilde{e}^{*},-\pi_{0}(t_{\pi}))_{\mathrm{I}}).
\end{align*}

\end{theorem}

The following theorem characterizes the equivalence class of $p$-structures
on double extensions.

\begin{theorem} The adapted isomorphism $\pi:L\rightarrow\widetilde{L}$
given in Eq. (4.7) defines a restricted isomorphism if and only if
\[
\pi_{0}(u^{[p]_{V}})=(\pi_{0}(u))^{[p]_{V}}+B_{V}(t_{\pi},u)^{p}\widetilde{u}_{0}
\]
and
\begin{align*}
\mathscr{\widetilde{P}}\circ\pi_{0}= & \gamma\mathscr{P}+B_{V}(t_{\pi},(\cdot)^{[p]_{V}})-B_{V}(t_{\pi},\cdot)^{p}\widetilde{m},\\
\widetilde{\xi}= & \gamma^{p-1}\xi,\\
\widetilde{l}= & \gamma^{p}(B_{V}(t_{\pi},a_{0})+\gamma l-\frac{\xi}{2\gamma}B_{V}(t_{\pi},t_{\pi}))+\mathscr{\widetilde{P}}(\pi_{0}(t_{\pi}))\\
 & +\frac{1}{2^{p}}B_{V}(t_{\pi},t_{\pi})^{p}\widetilde{m}-\sum_{i=1}^{p-1}\frac{1}{i}\Phi_{i}^{p}(\widetilde{e}^{*},-\pi_{0}(t_{\pi}))_{\mathrm{II}},\\
\widetilde{a}_{0}= & \gamma^{p}(\pi_{0}(a_{0})-\gamma^{-1}\xi\pi_{0}(t_{\pi}))+\frac{1}{2^{p}}B_{V}(t_{\pi},t_{\pi})^{p}\widetilde{u}_{0}+\pi_{0}(t_{\pi})^{[p]_{V}}\\
 & -\sum_{i=1}^{p-1}\frac{1}{i}\Phi_{i}^{p}(\widetilde{e}^{*},-\pi_{0}(t_{\pi}))_{\mathrm{I}},\\
\widetilde{m}= & \gamma^{-p}(\gamma m+B_{V}(t_{\pi},u_{0}))\\
\widetilde{u}_{0}= & \gamma^{-p}\pi_{0}(u_{0}).
\end{align*}
where $\Phi_{i}^{p}(\widetilde{e}^{*},-\pi_{0}(t_{\pi}))_{\mathrm{I}}$
and $\Phi_{i}^{p}(\widetilde{e}^{*},-\pi_{0}(t_{\pi}))_{\mathrm{II}}$
are as in Theorem 4.12.

\end{theorem}
\begin{proof}
($\Leftarrow$) Obviously.

($\Rightarrow$) For any $u\in V$, we have

\begin{align*}
\pi(u^{[p]_{L}}) & =\pi(u^{[p]_{V}}+\mathscr{P}(u)e)=\pi(u^{[p]_{V}})+\mathscr{P}(u)\pi(e)\\
 & =\pi_{0}(u^{[p]_{V}})+B_{V}(t_{\pi},u^{[p]_{V}})\widetilde{e}+\mathscr{P}(u)\gamma\widetilde{e}\\
 & =\pi_{0}(u^{[p]_{V}})+(B_{V}(t_{\pi},u^{[p]_{V}})+\mathscr{P}(u)\gamma)\widetilde{e}
\end{align*}
and
\begin{align*}
(\pi(u))^{[p]_{\widetilde{L}}}= & (\pi_{0}(u)+B_{V}(t_{\pi},u)\widetilde{e})^{[p]_{\widetilde{L}}}\\
= & (\pi_{0}(u))^{[p]_{\widetilde{L}}}+B_{V}(t_{\pi},u)^{p}\widetilde{e}^{[p]_{\widetilde{L}}}\\
= & (\pi_{0}(u))^{[p]_{V}}+\mathscr{\widetilde{P}}(\pi_{0}(u))\widetilde{e}+B_{V}(t_{\pi},u)^{p}(\widetilde{m}\widetilde{e}+\widetilde{u}_{0})\\
= & ((\pi_{0}(u))^{[p]_{V}}+B_{V}(t_{\pi},u)^{p}\widetilde{u}_{0})+(\mathscr{\widetilde{P}}(\pi_{0}(u))+B_{V}(t_{\pi},u)^{p}\widetilde{m})\widetilde{e}.
\end{align*}

It follows from $\pi(u^{[p]_{L}})=(\pi(u))^{[p]_{\widetilde{L}}}$
that
\[
\pi_{0}(u^{[p]_{V}})=(\pi_{0}(u))^{[p]_{V}}+B_{V}(t_{\pi},u)^{p}\widetilde{u}_{0}
\]
and
\[
\mathscr{\widetilde{P}}\circ\pi_{0}=\gamma\mathscr{P}+B_{V}(t_{\pi},(\cdot)^{[p]_{V}})-B_{V}(t_{\pi},\cdot)^{p}\widetilde{m}.
\]

Moreover, we have
\begin{align*}
\pi((e^{*})^{[p]_{L}}) & =\pi(a_{0}+le+\xi e^{*})=\pi(a_{0})+l\pi(e)+\xi\pi(e^{*})\\
 & =\pi_{0}(a_{0})+B_{V}(t_{\pi},a_{0})\widetilde{e}+l\gamma\widetilde{e}+\xi\gamma^{-1}(\widetilde{e}^{*}-\pi_{0}(t_{\pi})-\frac{1}{2}B_{V}(t_{\pi},t_{\pi})\widetilde{e})\\
 & =(\pi_{0}(a_{0})-\xi\gamma^{-1}\pi_{0}(t_{\pi}))+\xi\gamma^{-1}\widetilde{e}^{*}+(l\gamma+B_{V}(t_{\pi},a_{0})-\frac{1}{2}\xi\gamma^{-1}B_{V}(t_{\pi},t_{\pi}))\widetilde{e}
\end{align*}
and

\begin{align*}
 & (\pi(e^{*}))^{[p]_{\widetilde{L}}}\\
= & (\gamma^{-1}(\widetilde{e}^{*}-\pi_{0}(t_{\pi})-\frac{1}{2}B_{V}(t_{\pi},t_{\pi})\widetilde{e}))^{[p]_{\widetilde{L}}}\\
= & \gamma^{-p}(\widetilde{e}^{*}-\pi_{0}(t_{\pi}))^{[p]_{\widetilde{L}}}-\frac{1}{2^{p}}\gamma^{-p}B_{V}(t_{\pi},t_{\pi})^{p}\widetilde{e}^{[p]_{\widetilde{L}}}+\sum_{i=1}^{p-1}\gamma^{-p}s_{i}^{\widetilde{L}}(\widetilde{e}^{*}-\pi_{0}(t_{\pi}),-\frac{1}{2}B_{V}(t_{\pi},t_{\pi})\widetilde{e})\\
= & \gamma^{-p}(\widetilde{e}^{*})^{[p]_{\widetilde{L}}}-\gamma^{-p}\pi_{0}(t_{\pi})^{[p]_{\widetilde{L}}}+\gamma^{-p}\sum_{i=1}^{p-1}s_{i}^{\widetilde{L}}(\widetilde{e}^{*},-\pi_{0}(t_{\pi}))-\frac{1}{2^{p}}\gamma^{-p}B_{V}(t_{\pi},t_{\pi})^{p}\widetilde{e}^{[p]_{\widetilde{L}}}\\
= & \gamma^{-p}(\widetilde{a}_{0}+\widetilde{l}\widetilde{e}+\widetilde{\xi}\widetilde{e}^{*})-\gamma^{-p}\pi_{0}(t_{\pi})^{[p]_{V}}-\gamma^{-p}\widetilde{\mathscr{P}}(\pi_{0}(t_{\pi}))\widetilde{e}-\frac{1}{2^{p}}\gamma^{-p}B_{V}(t_{\pi},t_{\pi})^{p}(\widetilde{m}\widetilde{e}+\widetilde{u}_{0})\\
 & +\gamma^{-p}(\sum_{i=1}^{p-1}\frac{1}{i}\Phi_{i}^{p}(\widetilde{e}^{*},-\pi_{0}(t_{\pi}))_{\mathrm{I}}+\sum_{i=1}^{p-1}\frac{1}{i}\Phi_{i}^{p}(\widetilde{e}^{*},-\pi_{0}(t_{\pi}))_{\mathrm{II}}\widetilde{e})\\
= & (\gamma^{-p}\widetilde{a}_{0}-\gamma^{-p}\pi_{0}(t_{\pi})^{[p]_{V}}+\gamma^{-p}\sum_{i=1}^{p-1}\frac{1}{i}\Phi_{i}^{p}(\widetilde{e}^{*},-\pi_{0}(t_{\pi}))_{\mathrm{I}}-\frac{1}{2^{p}}\gamma^{-p}B_{V}(t_{\pi},t_{\pi})^{p}\widetilde{u}_{0})+\gamma^{-p}\widetilde{\xi}\widetilde{e}^{*}\\
 & +(\gamma^{-p}\widetilde{l}-\gamma^{-p}\widetilde{\mathscr{P}}(\pi_{0}(t_{\pi}))+\gamma^{-p}\sum_{i=1}^{p-1}\frac{1}{i}\Phi_{i}^{p}(\widetilde{e}^{*},-\pi_{0}(t_{\pi}))_{\mathrm{II}}-\frac{1}{2^{p}}\gamma^{-p}B_{V}(t_{\pi},t_{\pi})^{p}\widetilde{m})\widetilde{e}.
\end{align*}

It follows from $\pi((e^{*})^{[p]_{L}})=(\pi(e^{*}))^{[p]_{\widetilde{L}}}$
that

\begin{align*}
\widetilde{\xi}= & \gamma^{p-1}\xi,\\
\widetilde{l}= & \gamma^{p}(B_{V}(t_{\pi},a_{0})+\gamma l-\frac{\xi}{2\gamma}B_{V}(t_{\pi},t_{\pi}))+\mathscr{\widetilde{P}}(\pi_{0}(t_{\pi}))\\
 & +\frac{1}{2^{p}}B_{V}(t_{\pi},t_{\pi})^{p}\widetilde{m}-\sum_{i=1}^{p-1}\frac{1}{i}\Phi_{i}^{p}(\widetilde{e}^{*},-\pi_{0}(t_{\pi}))_{\mathrm{II}},
\end{align*}
and
\begin{align*}
\widetilde{a}_{0}= & \gamma^{p}(\pi_{0}(a_{0})-\gamma^{-1}\xi\pi_{0}(t_{\pi}))+\frac{1}{2^{p}}B_{V}(t_{\pi},t_{\pi})^{p}\widetilde{u}_{0}+\pi_{0}(t_{\pi})^{[p]_{V}}\\
 & -\sum_{i=1}^{p-1}\frac{1}{i}\Phi_{i}^{p}(\widetilde{e}^{*},-\pi_{0}(t_{\pi}))_{\mathrm{I}}.
\end{align*}

Finally, we have
\begin{align*}
0 & =\pi(e^{[p]_{L}})-\pi(e)^{[p]_{\widetilde{L}}}=\pi(me+u_{0})-(\gamma\widetilde{e})^{[p]_{\widetilde{L}}}\\
 & =m\gamma\widetilde{e}+\pi_{0}(u_{0})+B_{V}(t_{\pi},u_{0})\widetilde{e}-\gamma^{p}(\widetilde{m}\widetilde{e}+\widetilde{u}_{0})\\
 & =(\pi_{0}(u_{0})-\gamma^{p}\widetilde{u}_{0})+(m\gamma+B_{V}(t_{\pi},u_{0})-\gamma^{p}\widetilde{m})\widetilde{e}.
\end{align*}

It follows that $\widetilde{u}_{0}=\gamma^{-p}\pi_{0}(u_{0})$ and
$\widetilde{m}=\gamma^{-p}(m\gamma+B_{V}(t_{\pi},u_{0}))$.
\end{proof}

\section{Examples of Double Extensions}

In Sec. 2, it has been pointed out that the cohomology structures
defined in \citep{MS} and \citep{AEM} can not be used in the study
of the double extensions of Hom-Lie algebras. Therefore, we can not
give the examples of the double extensions of Hom-Lie algebras with
the corresponding cohomology theory. We will give several examples
of double extensions with other methods.

\begin{example} Suppose that $\mathrm{char}(\mathbb{F})=2$. Consider
the Heisenberg Lie algebra $\mathcal{H}$ which is spanned by $x,y,z$.
Let $V=\mathcal{H\oplus\mathcal{H}^{*}}$ with the bracket $[\cdot,\cdot]_{V}:V\times V\rightarrow V$
defined by
\[
[x,y]_{V}=z,\:[x,z^{*}]_{V}=y^{*},\:[y,z^{*}]_{V}=x^{*}
\]
and a bilinear form $B_{V}:V\times V\rightarrow\mathbb{F}$ defined
by
\[
B_{V}(x,x^{*})=1,\:B_{V}(y,y^{*})=1,\:B_{V}(z,z^{*})=1.
\]
Then the triple $(V,[\cdot,\cdot]_{V},B_{V})$ is a quadratic Lie
algebra (i.e. a Lie algebra with a nondegenerate invariant symmetric
bilinear form). A $2$-structure $[2]_{V}$ on $V$ is defined by
(for any $l_{1},l_{2},l_{3},m_{1},m_{2},m_{3}\in\mathbb{F}$)
\[
(l_{1}x+l_{2}y+l_{3}z+m_{1}x^{*}+m_{2}y^{*}+m_{3}z^{*})^{[2]_{V}}=((l_{3})^{2}+l_{1}l_{2})z+l_{1}m_{3}y^{*}+l_{2}m_{3}x^{*},
\]
which implies that
\[
x^{[2]_{V}}=y^{[2]_{V}}=0,\:z^{[2]_{V}}=z,\:(x^{*})^{[2]_{V}}=(y^{*})^{[2]_{V}}=(z^{*})^{[2]_{V}}=0.
\]

Furthermore, we define a derivation $\mathscr{D}:V\rightarrow V$
as follows:
\[
\mathscr{D}(x)=x,\:\mathscr{D}(y)=y,\:\mathscr{D}(z)=0,\:\mathscr{D}(x^{*})=x^{*},\:\mathscr{D}(y^{*})=y^{*},\:\mathscr{D}(z^{*})=0.
\]
It's easy to prove that $\mathscr{D}$ is a restricted derivation
with respect to $[2]_{V}$ on $V$ such that $B_{V}$ is $\mathscr{D}$-invariant.
If $\xi=1$ and $a_{0}=z$, then $\mathscr{D}$ satisfies Eq. (2.2),
that is, $\mathscr{D}$ has $2$-property.

Let $L=\mathscr{E}^{*}\oplus V\oplus\mathscr{E}$. According to Theorem
3.1, the bracket $[\cdot,\cdot]:L\times L\rightarrow L$ is defined
by
\[
[x,y]=z,\:[x,x^{*}]=e,\:[x,z^{*}]=y^{*},\:[y,y^{*}]=e,\:[y,z^{*}]=x^{*},
\]
\[
[e^{*},x]=x,\:[e^{*},y]=y,\:[e^{*},x^{*}]=x^{*},\:[e^{*},y^{*}]=y^{*},
\]
the linear map $\alpha:L\rightarrow L$ is defined by
\[
\alpha(v)=v,\:\forall v\in V,\:\alpha(e^{*})=e^{*}+\lambda_{0}e,\:\alpha(e)=e
\]
and the bilinear form $B:L\times L\rightarrow\mathbb{F}$ is defined
by
\[
B(u,v)=B_{V}(u,v),\:B(v,e^{*})=B(v,e)=0,\:\forall u,v\in V,
\]
\[
B(e^{*},e)=1,\:B(e,e)=0.
\]

Then the quadruple $(L,[\cdot,\cdot],\alpha,B)$ is a quadratic Hom-Lie
algebra, which is the double extension of $V$.

In addition, we define the map $\mathscr{P}:V\rightarrow\mathbb{F}$
as follows (for any $l_{1},l_{2},l_{3},m_{1},m_{2},m_{3}\in\mathbb{F}$):
\[
\mathscr{P}(l_{1}x+l_{2}y+l_{3}z+m_{1}x^{*}+m_{2}y^{*}+m_{3}z^{*})=l_{1}m_{1}+l_{2}m_{2}.
\]
Then $\mathscr{P}$ satisfies Eqs. (3.5) and (3.6). Therefore, we
obtain a $2$-structure $[2]_{L}:L\rightarrow L$ defined by

\[
v^{[2]_{L}}=v^{[2]_{V}}+\mathscr{P}(v)e,\:\forall v\in V,\:(e^{*})^{[2]_{L}}=z+le+e^{*},\:e^{[2]_{L}}=me+z.
\]

\end{example}

To provide more examples, we first introduce the following results
as preparation.

\begin{lemma} (see \citep{BM,BM2}) Let $(\mathfrak{g},[\cdot,\cdot],B)$
be a quadratic Lie algebra and $\alpha\in\mathrm{End}_{s}(\mathfrak{g},B)$.

(1) The quadruple $\mathfrak{g}_{\alpha}=(\mathfrak{g},[\cdot,\cdot]_{\alpha},\alpha,B_{\alpha})$
be a quadratic Hom-Lie algebra, where $[x,y]_{\alpha}=\alpha\circ[x,y]$,
$B_{\alpha}(x,y)=B(\alpha(x),y)$, for any $x,y\in\mathfrak{g}$.

(2) Let $[p]:\mathfrak{g}\rightarrow\mathfrak{g}$ be a $p$-mapping
on Lie algebra $(\mathfrak{g},[\cdot,\cdot],B)$, then $[p]_{\alpha}:\mathfrak{g}\rightarrow\mathfrak{g}$
defined by $x^{[p]_{\alpha}}=\alpha^{p-1}(x^{[p]})$, $\forall x\in\mathfrak{g}$
is a $p$-structure on $\mathfrak{g}_{\alpha}$.

\end{lemma}

\begin{proposition} Suppose that $\mathrm{char}(\mathbb{F})>2$.
Let $(\mathfrak{g},[\cdot,\cdot],B)$ be a quadratic Lie algebra,
$[p]:\mathfrak{g}\rightarrow\mathfrak{g}$ be a $p$-mapping and $\alpha\in\mathrm{End}_{s}(\mathfrak{g},B)$
with $\alpha^{2}=\mathrm{id}$. Let $\mathfrak{g}_{\alpha}=(\mathfrak{g},[\cdot,\cdot]_{\alpha},\alpha,B_{\alpha})$
be a quadratic Hom-Lie algebra and $[p]_{\alpha}:\mathfrak{g}\rightarrow\mathfrak{g}$
a $p$-structure on $\mathfrak{g}_{\alpha}$ as defined in Lemma 5.2.
If $\mathscr{D}\in\mathrm{Der}(\mathfrak{g})$ is restricted with
respect to $[p]$ such that $\alpha\circ\mathscr{D}=\mathscr{D}\circ\alpha$
and $\mathscr{D}$ has $p$-property, then $\mathscr{D}_{\alpha}:=\alpha\circ\mathscr{D}\in\mathrm{Der}_{\alpha}(\mathfrak{g}_{\alpha})$
is restricted with respect to $[p]_{\alpha}$ and $\mathscr{D}_{\alpha}$
also has $p$-property.

\end{proposition}
\begin{proof}
Clearly, $\mathscr{D}_{\alpha}\circ\alpha=\alpha\circ\mathscr{D}_{\alpha}.$
For any $x,y\in\mathfrak{g}_{\alpha}$, we have
\begin{align*}
\mathscr{D}_{\alpha}([x,y]_{\alpha}) & =\alpha\circ\mathscr{D}(\alpha([x,y]))=\alpha^{2}\circ\mathscr{D}([x,y])\\
 & =[\mathscr{D}(x),y]+[x,\mathscr{D}(y)]\\
 & =[\alpha^{2}\circ\mathscr{D}(x),\alpha^{2}(y)]+[\alpha^{2}(x),\alpha^{2}\circ\mathscr{D}(y)]\\
 & =\alpha([\alpha\circ\mathscr{D}(x),\alpha(y)])+\alpha([\alpha(x),\alpha\circ\mathscr{D}(y)])\\
 & =[\mathscr{D}_{\alpha}(x),\alpha(y)]_{\alpha}+[\alpha(x),\mathscr{D}_{\alpha}(y)]_{\alpha}.
\end{align*}

Therefore, $\mathscr{D}_{\alpha}:=\alpha\circ\mathscr{D}\in\mathrm{Der}_{\alpha}(\mathfrak{g}_{\alpha})$.

In addition, it follows from $\mathscr{D}(x^{[p]})=(\mathrm{ad}x)^{p-1}(\mathscr{D}(x))$,
$\forall x\in\mathfrak{g}$ that

\begin{align*}
\mathscr{D}_{\alpha}(x^{[p]_{\alpha}}) & =\alpha\circ\mathscr{D}(\alpha^{p-1}(x^{[p]}))=\alpha^{p}\circ\mathscr{D}(x^{[p]})\\
 & =\alpha^{p}\circ(\mathrm{ad}x)^{p-1}(\mathscr{D}(x))=\alpha^{p}\circ[x,[x,\cdot\cdot\cdot,[x,\mathscr{D}(x)]]]\\
 & =[\alpha^{p-1}(x),\alpha^{p-1}([x,\cdot\cdot\cdot,[x,\mathscr{D}(x)]])]_{\alpha}\\
 & =[\alpha^{p-1}(x),[\alpha^{p-2}(x),\cdot\cdot\cdot,[\alpha(x),\alpha\circ\mathscr{D}(x)]_{\alpha}]_{\alpha}]_{\alpha}\\
 & =[\alpha^{p-1}(x),[\alpha^{p-2}(x),\cdot\cdot\cdot,[\alpha(x),\mathscr{D}_{\alpha}(x)]_{\alpha}]_{\alpha}]_{\alpha}.
\end{align*}
It implies that $\mathscr{D}_{\alpha}$ is a restricted derivation
with respect to $[p]_{\alpha}$.

Since $\mathscr{D}$ has $p$-property, there exists $\xi\in\mathbb{F}$
and $a_{0}\in\mathfrak{g}$ such that $\mathscr{D}^{p}=\xi\mathscr{D}+\mathrm{ad}a_{0}$
and $\mathscr{D}(a_{0})=0$. Therefore,
\begin{align*}
\mathscr{D}_{\alpha}^{p} & =(\alpha\circ\mathscr{D})^{p}=\alpha^{p}\circ\mathscr{D}^{p}=\alpha^{p-1}\circ\alpha\circ\mathscr{D}^{p}=\alpha(\xi\mathscr{D}+\mathrm{ad}a_{0})\\
 & =\xi\alpha\circ\mathscr{D}+\alpha\circ[a_{0},\cdot]=\xi\mathscr{D}_{\alpha}+[a_{0},\cdot]_{\alpha}=\xi\mathscr{D}_{\alpha}\circ\alpha^{p-1}+[a_{0},\alpha^{p-1}(\cdot)]_{\alpha}
\end{align*}
and $\mathscr{D}_{\alpha}(a_{0})=\alpha\circ\mathscr{D}(a_{0})=0$,
which implies that $\mathscr{D}_{\alpha}$ has $p$-property.

The proof is complete.
\end{proof}
\begin{example} Suppose that $\mathrm{char}(\mathbb{F})=3$. Consider
an involutive quadratic Hom-Lie algebra $\mathfrak{psl}(3)_{\alpha}$,
which is obtained by twisting Lie algebra $\mathfrak{psl}(3)$ \citep{BBH,BKLS}.
Details are as follows:

For 3-dimensional vector space $V$, let $\mathfrak{sl}(3):=\{A\in\mathrm{\mathfrak{gl}}(V)|\mathrm{tr}(A)=0\}$
denote the special linear Lie algebra and $\mathfrak{s}$ a Lie algebra
of scalar matrices. Then
\[
\mathfrak{psl}(3)=\mathfrak{sl}(3)/\mathfrak{s}
\]
is a Lie algebra whose elements are called projective transformations.

Denote by $x_{i}$ the positive elements of the Chevalley basis (see
\citep{CCLL}) and $y_{i}$ the corresponding negative ones. Then
$\{x_{1},x_{2},x_{3}=[x_{1},x_{2}],y_{1},y_{2},y_{3}=[y_{1},y_{2}]\}$
is the set of generators of Lie algebra $\mathfrak{psl}(3)$. In the
ordered basis $e_{1}=[x_{1},y_{1}],e_{2}=x_{1},e_{3}=x_{2},e_{4}=x_{3},e_{5}=y_{1},e_{6}=y_{2},e_{7}=y_{3}$
of $\mathfrak{psl}(3)$, the nondegenerate symmetric invariant bilinear
form $B_{\mathfrak{psl}(3)}:\mathfrak{psl}(3)\times\mathfrak{psl}(3)\rightarrow\mathbb{F}$
has the Gram matrix

\[
\mathscr{B}_{\mathfrak{psl}(3)}=\left(\begin{array}{ccc}
-1 & 0 & 0\\
0 & 0 & I_{2,1}\\
0 & I_{2,1} & 0
\end{array}\right),
\]
where $I_{r,s}=\mathrm{diag}(1,\cdot\cdot\cdot,1,-1,\cdot\cdot\cdot,-1)$
with $r$-many 1s and $s$-many (-1)s. The 3-mapping on $\mathfrak{psl}(3)$
is defined by
\[
e_{1}^{[3]}=e_{1},\:e_{i}^{[3]}=0,\:\mathrm{for}\;2\leq i\leq7.
\]

Let $\mathscr{D}$ be a derivation in one of the following forms
\[
\mathscr{D}_{1}=y_{1}\otimes\hat{x_{3}}+y_{3}\otimes\hat{x_{1}},
\]
\begin{equation}
\mathscr{D}_{2}=2x_{1}\otimes\hat{x_{2}}+y_{2}\otimes\hat{y_{1}},
\end{equation}
\[
\mathscr{D}_{3}=x_{1}\otimes\hat{x_{1}}+x_{3}\otimes\hat{x_{3}}+2y_{1}\otimes\hat{y_{1}}+2y_{3}\otimes\hat{y_{3}}.
\]
Then $\mathscr{D}$ is restricted with respect to the $3$-mapping
$[3]$ and the bilinear form $B_{\mathfrak{psl}(3)}$ is $\mathscr{D}$-invariant
for any $\mathscr{D}$ in (5.1).

Let $\alpha:\mathfrak{psl}(3)\rightarrow\mathfrak{psl}(3)$ be a map
defined as follows:
\[
\alpha(e_{i})=e_{i},\:i=1,4,7;\:\alpha(e_{j})=-e_{j},\:j=2,3,5,6.
\]
A direct compution shows that $\alpha\in\mathrm{End}_{s}(\mathfrak{psl}(3),B_{\mathfrak{psl}(3)})$
and $\alpha^{2}=\mathrm{id}$. According to Lemma 5.2, the quadruple
$\mathfrak{psl}(3)_{\alpha}=(\mathfrak{psl}(3),[\cdot,\cdot]_{\alpha},\alpha,B_{\alpha})$
is an involutive quadratic Hom-Lie algebra, where $[x,y]_{\alpha}=\alpha([x,y])$,
$B_{\alpha}(x,y)=B_{\mathfrak{psl}(3)}(\alpha(x),y)$, $\forall x,y\in\mathfrak{psl}(3)$.
Let $[3]_{\alpha}=[3]:\mathfrak{psl}(3)_{\alpha}\rightarrow\mathfrak{psl}(3)_{\alpha};x\mapsto x^{[3]_{\alpha}}$,
then $[3]_{\alpha}$ is a $3$-structure on $\mathfrak{psl}(3)_{\alpha}$.
$\mathscr{D}_{\alpha}:=\alpha\circ\mathscr{D}$ is a restricted derivation
with respect to $[3]_{\alpha}$.

Consider the $3$-mapping on double extensions of Lie algebra $\mathfrak{psl}(3)$.
The map $\mathscr{P}:\mathfrak{psl}(3)\rightarrow\mathbb{F}$ should
satisfy (for any $x,y\in\mathfrak{psl}(3)$ and any $t\in\mathbb{F}$):

\begin{align*}
\mathscr{P}(tx) & =t^{3}\mathscr{P}(x),\\
\mathscr{P}(x+y) & =\mathscr{P}(x)+\mathscr{P}(y)+B_{\mathfrak{psl}(3)}(\mathscr{D}(y),[y,x])+\frac{1}{2}B_{\mathfrak{psl}(3)}(\mathscr{D}(x),[y,x]).
\end{align*}
On the other hand, consider the $3$-structure on double extensions
of Hom-Lie algebra $\mathfrak{psl}(3)_{\alpha}$. The map $\mathscr{P}:\mathfrak{psl}(3)_{\alpha}\rightarrow\mathbb{F}$
should satisfy (for any $x,y\in\mathfrak{psl}(3)_{\alpha}$ and any
$t\in\mathbb{F}$):

\begin{align*}
\mathscr{P}(tx) & =t^{3}\mathscr{P}(x),\\
\mathscr{P}(x+y) & =\mathscr{P}(x)+\mathscr{P}(y)+B_{\alpha}(\mathscr{D}_{\alpha}(\alpha(y)),[y,x]_{\alpha})+\frac{1}{2}B_{\alpha}(\mathscr{D}_{\alpha}(\alpha(x)),[y,x]_{\alpha})\\
 & =\mathscr{P}(x)+\mathscr{P}(y)+B_{\mathfrak{psl}(3)}(\mathscr{D}(y),[y,x])+\frac{1}{2}B_{\mathfrak{psl}(3)}(\mathscr{D}(x),[y,x]).
\end{align*}
Therefore, the conditions for the map $\mathscr{P}$ are the same
in these two cases.

Let $x_{0}=0$ and the table for double extensions of $\mathfrak{psl}(3)_{\alpha}$
is as follows:
\begin{center}
\begin{tabular}{|c|c|c|c|c|}
\hline
Derivation & {\footnotesize{}$\mathscr{P}(x)$ $(x=\sum_{i=1}^{7}\lambda_{i}e_{i})$} & $\xi$ & $a_{0}$ & Double Extension\tabularnewline
\hline
\hline
$\mathscr{D}_{1}$ & $\lambda_{3}^{2}\lambda_{7}+2\lambda_{1}\lambda_{5}\lambda_{3}+\lambda_{4}\lambda{}_{5}^{2}$ & 0 & 0 & {\footnotesize{}$\mathfrak{\widetilde{gl}}(3)_{\alpha}$}\tabularnewline
\hline
$\mathscr{D}_{2}$ & $\lambda{}_{4}^{2}\lambda_{6}+2\lambda_{1}\lambda_{2}\lambda_{4}+2\lambda_{2}^{2}\lambda_{3}$ & 0 & 0 & {\footnotesize{}$\mathfrak{\hat{gl}}(3)_{\alpha}$}\tabularnewline
\hline
$\mathscr{D}_{3}$ & $\lambda_{1}\lambda_{2}\lambda_{5}+\lambda_{4}\lambda_{6}\lambda_{5}+2\lambda_{2}\lambda_{3}\lambda_{7}+\lambda_{1}\lambda_{4}\lambda_{7}$ & 1 & 0 & {\footnotesize{}$\mathfrak{gl}(3)_{\alpha}$}\tabularnewline
\hline
\end{tabular}
\par\end{center}

\end{example}

\begin{remark} As we can see, Example 5.1 is trivial, that is, the
twist map $\alpha=\mathrm{id}$ in it and the  involutive quadratic
Hom-Lie algebra $\mathfrak{psl}(3)_{\alpha}$ in Example 5.4 is obtained
by twisting Lie algebra $\mathfrak{psl}(3)$. That is to say, neither
of the two examples directly give double extensions of an involutive
quadratic Hom-Lie algebra. That is because it's difficult to find
suitable outer derivations for a given restricted involutive quadratic
Hom-Lie algebra. Inspired by the corresponding research in Lie algebras,
it is known that the (restricted) cohomology theory of (restricted)
Hom-Lie algebras may solve the difficulty. In order to give more examples,
we will define cohomology structures on Hom-Lie algebras in prime
characteristic and restricted cohomology structures on restricted
Hom-Lie algebras in the following research.

\end{remark}


\begin{thebibliography}{99}
\bibitem{AEM} Ammar, F., Ejbehi, Z. and Makhlouf, A., Cohomology
and deformations of Hom-algebras. J. Lie Theory, 2011, 21(4): 813-836.

\bibitem{ABB} Albuquerque, H., Barreiro, E. and Benayadi, S., Quadratic
Lie superalgebras with a reductive even part. J. Pure Appl. Algebra,
2009, 213: 724-731.

\bibitem{ABB2} Albuquerque, H., Barreiro, E. and Benayadi, S., Odd
quadratic Lie superalgebras. J. Geom. Phys., 2010, 60(2): 230-250.

\bibitem{ARS} Alvarez, M. A., Rodríguez-Vallarte, M. C. and Salgado,
G., Deformation theory of contact Lie algebras as double extensions.
Proc. Amer. Math. Soc., 2021, 149(5): 1827-1836.

\bibitem{BMPZ} Bahturin, Y., Mikhalev, A., Petrogradski, V. M. and
Zaicev, M., Infinite-dimensional Lie superalgebras. Walter de Gruyter,
Berlin, New York, 1992.

\bibitem{BB} Benamor H. and Benayadi S., Double extension of quadratic
Lie superalgebras. Commun. Algebra, 1999, 27 (1): 67-88.

\bibitem{Be} Benayadi, S., Quadratic Lie superalgebras with completely
reductive action of the even part on the odd part. J. Algebra, 2000,
223: 344-366.

\bibitem{Be2} Benayadi, S., Socle and some invariants of quadratic
Lie superalgebras. J. Algebra, 2003, 261: 245-291.

\bibitem{BBH} Benayadi, S., Bouarroudj, S. and Hajli, M., Double
extensions of restricted Lie (super)algebras. Arnold Math. J., 2020,
6(2): 231-269.

\bibitem{BB2} Benayadi, S. and Bouarroudj, S., Manin triples and
non-degenerate anti-symmetric bilinear forms on Lie superalgebras
in characteristic 2. J. Algebra, 2023, 614: 199-250.

\bibitem{BM} Benayadi, S. and Makhlouf, A., Hom-Lie algebras with
symmetric invariant nondegenerate bilinear forms. J. Geom. Phys.,
2014, 76: 38-60.

\bibitem{BeBou} Benayadi, S. and Bouarroudj, S., Double extensions
of Lie superalgebras in characteristic 2 with nondegenerate invariant
symmetric bilinear forms. J. Algebra, 2018, 510: 141-179.

\bibitem{BKLS}Bouarroudj, S., Krutov, A., Leites, D. and Shchepochkina,
I., Non-degenerate invariant (super)symmetric bilinear forms on simple
Lie (super)algebras. Algebr. Represent. Theory, 2018, 21(5): 897--941.

\bibitem{BLS} Bouarroudj, S., Leites, D. and Shang, J., Computer-aided
study of double extensions of restricted Lie superalgebras preserving
the nondegenerate closed 2-forms in characteristic 2. Exp. Math.,
2022, 31(2): 676-688.

\bibitem{BEM} Bouarroudj, S., Ehret, Q. and Maeda, Y., Symplectic
double extensions for restricted quasi-Frobenius Lie (super)algebras.
SIGMA Symmetry Integrability Geom. Methods Appl., 2023, 19(70): 1-29.

\bibitem{BM2} Bouarroudj, S. and Makhlouf, A., Hom-Lie superalgebras
in characteristic 2. arXiv: 2210.08986 (2022).

\bibitem{CCLL} Chapovalov, D., Chapovalov, M., Lebedev, A. and Leites,
D., The classification of almost affine (hyperbolic) Lie superalgebras.
J. Nonlinear Math. Phys., 2010, 17(1): 103--161.

\bibitem{DL} Dokas, I. and Loday, J. L., On restricted Leibniz algebras.
Comm. Algebra, 2006, 34: 4467-4478.

\bibitem{EF} Evans, T. J. and Fuchs, D., A complex for the cohomology
of restricted Lie algebras. J. fixed point theory appl., 2008, 3:
159-179.

\bibitem{Farn} Farnsteiner, R., Note on Frobenius extensions and
restricted Lie superalgebras. J. Pure Appl. Algebra, 1996, 108: 241-256.

\bibitem{FS} Favre, G. and Santharoubane, L.J., Symmetric, invariant,
non-degenerate bilinear form on a Lie algebra. J. Algebra, 1987, 105:
451-464.

\bibitem{Fel} Feldvoss, J., On the cohomology of restricted Lie algebras.
J. Algebra, 1991, 19: 2856-2906.

\bibitem{GC} Guan, B. and Chen, L., Restricted Hom-Lie algebras.
Hacet. J. Math. Stat., 2015, 44(4): 823-837.

\bibitem{HLS} Hartwig, J.T., Larsson, D. and Silvestrov, S., Deformations
of Lie algebras using $\sigma$-derivations. J. Algebra, 2006, 295(2):
314-361.

\bibitem{Hoc} Hochschild, G., Cohomology of restricted Lie algebras.
Amer. J. Math., 1954, 76: 591-603.

\bibitem{Hod} Hodge, T. L., Lie triple system, restricted Lie triple
system and algebraic groups. J. Algebra, 2001, 244: 533-580.

\bibitem{Jac} Jacobson, N. , Restricted Lie algebras of characteristic
$p$. Trans. Amer. Math. Soc., 1941, 50: 15-25.

\bibitem{Jac2}Jacobson, N., Lie algebras. Interscience, New York,
1962.

\bibitem{JA} Jean-Michel, D. and Alberto, M., Algèbres de Lie kählériennes
et double extension. (French), J. Algebra, 1996, 185(3): 774-795.

\bibitem{MS} Makhlouf, A. and Silvestrov, S., Notes on 1-parameter
formal deformations of Hom-associative and Hom-Lie algebras. Forum
Math., 2010, 22(4): 715--739.

\bibitem{MGC} Mao, D., Guan, B. and Chen L., Modular structure theory
on Hom-Lie algebras. (2023), https://www.researchgate.net/publication/367389633

\bibitem{MR} Medina, A. and Revoy, P., Algèbres de Lie et produit
scalaire invariant. Ann. Sci. Éc. Norm. Supér. (4), 1985, 18(3): 553-561.

\bibitem{RS} Rodríguez-Vallarte, M. C. and Salgado, G., 5-dimensional
indecomposable contact Lie algebras as double extensions. J. Geom.
Phys., 2016, 100: 20-32.

\bibitem{RS2} Rodríguez-Vallarte, M. C. and Salgado, G., Geometric
structures on Lie algebras and double extensions. Proc. Amer. Math.
Soc., 2018, 146(10): 4199-4209.

\bibitem{Sha} Shaqaqha, S., Restricted Hom-Lie superalgebras. Jordan
Journal of Mathematics and Statistics, 2019, 12(2): 233-252.

\bibitem{She} Sheng, Y., Representations of Hom-Lie algebras. Algebr.
Represent. Theory, 2012, 15: 1081-1098.

\bibitem{SF}Strade, H., Farnsteiner, R., Modular Lie Algebras and
Their Representations. Marcel Dekker, New York, 1988.

\bibitem{SLGW} Sun, L., Liu, W., Gao, X. and Wu, B., Restricted envelopes
of Lie superalgebras. Algebra Colloq., 2015, 22(2): 309-320.

\bibitem{WZ} Wang, Y. and Zhang, Y., A new definition of restricted
Lie superalgebras. Chinese Sci. Bull., 2000, 45(4): 316-321.
\end{thebibliography}
\end{document}